\documentclass[12pt]{article}
\usepackage{amsfonts,amsthm}
\usepackage{mathrsfs,bbm,bm}
\usepackage{amsmath,rotating,amssymb}
\usepackage{graphicx,graphics,subfigure}
\usepackage{cite,color}
\usepackage{arydshln}
\usepackage{booktabs,mathtools,stmaryrd}
\usepackage{multirow,makecell}

\newcommand{\sq}{\varepsilon}
\newcommand{\ssq}{\sqrt{\varepsilon}}
\newcommand{\sfq}{\sqrt[4]{\varepsilon}}
\newcommand{\invssq}{\varepsilon^{-\frac12}}

\newcommand{\uN}{\mathbbm{u}}
\newcommand{\pN}{\mathbbm{p}}
\newcommand{\qN}{\mathbbm{q}}
\newcommand{\wN}{\bm{\mathbbm{w}}}
\newcommand{\wph}{\bm{\mathbbm{W}}}
\newcommand{\uph}{\mathbbm{U}}
\newcommand{\pph}{\mathbbm{P}}
\newcommand{\qph}{\mathbbm{Q}}

\newcommand{\vph}{\bm{\mathbbm{z}}}
\newcommand{\vphu}{\mathbbm{v}}
\newcommand{\vphq}{\mathbbm{r}}
\newcommand{\vphp}{\mathbbm{s}}

\newcommand{\dual}[2]{\left\langle#1,#2\right\rangle}

\newcommand{\norm}   [1] {\left\Vert#1\right\Vert}

\newcommand{\enorm}   [1] {\interleave #1\interleave_{E}}

\newcommand{\ba}   [1] {\interleave #1\interleave_{B}}

\newcommand{\jump}   [1] {[\![#1]\!]}

\newcommand{\Bln}[2]{B(#1;#2)}
\newcommand{\Ton}[2]{\mathcal{T}_1(#1;#2)}
\newcommand{\Ttw}[2]{\mathcal{T}_2(#1;#2)}
\newcommand{\Tth}[2]{\mathcal{T}_3(#1;#2)}
\newcommand{\Tfr}[2]{\mathcal{T}_4(#1;#2)}

\newcommand{\mcs}{\Theta}
\newcommand{\spc}{\mathcal{V}_N}

\newcommand{\ppri}{\psi^\prime}

\newcommand{\ofh}{N/4}
\newcommand{\tfh}{3N/4}

\newcommand{\parep}{\varrho}
\newcommand{\sparep}{\sqrt{\varrho}}

\newcommand{\kij}   {K_{ij}}
\newcommand{\xl}   {\Omega_{XL}}
\newcommand{\xm}   {\Omega_{XM}}
\newcommand{\xr}   {\Omega_{XR}}

\newcommand{\dx}{\partial_x}
\newcommand{\dy}{\partial_y}

\renewcommand{\theequation} {\arabic{section}.\arabic{equation}}

\newtheorem  {proposition} {\hspace{15pt}Proposition}[section]
\newtheorem  {theorem}    {\hspace{15pt} Theorem}[section]

\newtheorem  {lemma}     {\hspace{15pt} Lemma}[section]

\title{\bf Local discontinuous Galerkin method
	on layer-adapted meshes for singularly perturbed reaction--diffusion problems in two dimensions
}
\author{
Yanjie Mei\footnotemark[1],\quad
Yao Cheng\footnotemark[2],\quad
Sulei Wang\footnotemark[2],
\quad Zhijie Xu\footnotemark[2]
}

\begin{document}

\maketitle

\renewcommand{\thefootnote}{\fnsymbol{footnote}}
\footnotetext[1]
{International Education School,
	Suzhou University of Science and Technology,
	215009 Suzhou, Jiangsu Province, P.~R.~China.
	E-mail: yjmei@post.usts.edu.cn.
}
\footnotetext[2]
{School of Mathematical Sciences,
Suzhou University of Science and Technology,
Suzhou 215009, Jiangsu Province, P.~R.~China.
E-mail: ycheng@usts.edu.cn (corresponding author),
slwang@post.usts.edu.cn, zjxu@post.usts.edu.cn.
}


\begin{center}
\small
\begin{minipage}{0.8\textwidth}
\textbf{Abstract.}
We analyse the local discontinuous Galerkin (LDG) method for two-dimensional singularly perturbed reaction--diffusion problems. A class of layer-adapted meshes, including Shishkin- and Bakhvalov-type meshes, is discussed within a general framework. Local projections and their approximation properties on anisotropic meshes are used to derive error estimates for energy and ``balanced"norms. Here, the energy norm is naturally derived from the bilinear form of LDG formulation and the ``balanced" norm is artifically introduced to capture the boundary layer contribution. We establish a uniform convergence of order $k$ for the LDG method using the balanced norm with the local weighted $L^2$ projection as well as an optimal convergence of order $k+1$ for the energy norm using the local Gauss--Radau projections. Numerical experiments are presented.
\medskip


\medskip

\textbf{AMS.}
65N30, 65N15,65N12

\end{minipage}
\end{center}

\section{Introduction}
\label{intro}
Over the past few decades,
singularly perturbed problems have attracted considerable attention in the scientific community. 
Such problems arise in many applications,
including the modelling of viscous fluid flows, semiconductor devices, and more.
For reaction-diffusion problems, difficulties arise owing to the presence of boundary layers in the solution. Unless the meshes are sufficiently refined, traditional finite-difference or finite-element methods on uniform or quasi-uniform meshes yield oscillatory and inaccurate numerical solutions. Consequently, three common approaches have been proposed in the literature. The first is to use traditional numerical methods on strongly refined, layer-adapted meshes, such as the Shishkin-type (S-type) or Bakhvalov-type (B-type) mesh \cite{Bakhvalov1969,Linss2001,Linss2003book,Shishkin1990}. Various parameter-uniform convergence results have been established in this way;
notably, the order of convergence and error constant are independent of
the singular perturbation parameters.
The second approach is to use a stabilised numerical method, such as the streamline diffusion finite-element method, interior-penalty discontinuous Galerkin method, or local discontinuous Galerkin (LDG) method \cite{Johnson1984,Cheng2015,Cheng2017}; well-behaved local error estimates have been investigated using uniform or quasi-uniform meshes. The third approach is to combine the aforementioned stabilised numerical methods with layer-adapted meshes. From a practical perspective, 
the third approach is preferable because it is more stable and less sensitive to the choice of transition point on the layer-adapted mesh.

The LDG method is a form of finite-element method; it was first proposed 
as a generalisation of the discontinuous Galerkin (DG) method for a 
convection--diffusion problem \cite{Cockburn:Shu:LDG}. Later, it was applied to solve the purely elliptic problem \cite{Cockburn2001} and other higher-order partial differential equations\cite{Xu2010}. Because the LDG method shares many advantages of the DG methods and can effectively simulate the acute change of a singular solution, it is particularly suited to solving singularly perturbed problems. For example, Cheng et al. performed  
double-optimal local error estimates for two explicit, fully discrete LDG methods on quasi-uniform meshes \cite{Cheng2015,Cheng2017}. Xie et al. established uniform convergence and super-convergence analyses of the LDG method on a standard Shishkin mesh \cite{Zhu:2dMC,Zhu:1dCMS,Xie2010MC}. However, few results have been established for the LDG method on general S-type or B-type meshes.

Recently, we analysed the LDG method on several S-type meshes and a B-type mesh for singularly perturbed convection--diffusion problems. Robust error estimates were derived from the energy norm \cite{Cheng2020}. 
However, the reaction--diffusion case remains unexplored.
Despite its simpler appearance, reaction-dominated diffusion without convection differs from convection--diffusion in the following three respects:

\begin{itemize}
	\item For singularly perturbed reaction--diffusion problems, the boundary layer structure is considerably more complicated. As a result, the regularity of the solution is complex. This adds many difficulties to the theoretical analysis, such as in the construction of layer-adapted meshes and the estimates of various approximation errors.
	\item If a purely alternating numerical flux is employed in the LDG method, we have no interior boundary jump term in the energy norm. 
	Therefore, it is possible to establish an optimal convergence of order $k+1$ for the LDG method in the energy norm. To highlight the influence of the singularly perturbed parameter, we perform a more elaborate analysis
	for the two-dimensional Gauss--Radau projections on anisotropic meshes.
	\item In the reaction--diffusion region, the energy norm is inadequate because it cannot reflect the contribution of the boundary layer component. A balanced norm was introduced in \cite{Lin2012} to address this problem. To date, balanced-norm error estimates are available for the Galerkin finite-element method (FEM) \cite{Roos2015,Roos2016}, mixed FEM \cite{Lin2012}, and hp-FEM  \cite{Melenk2018}, but not for the LDG method.  
	For the first time, we establish the uniform convergence of the LDG method for the balanced norm.
\end{itemize}

The remainder of this paper is organised as follows: in Section 2, we describe the LDG method; in Section 3, we introduce a class of layer-adapted meshes and state some elementary lemmas for them;
in Section 4, we establish uniform convergence for the balanced and energy norms; in Section 5, we present numerical experiments; and in Section 6, we mention a convergence result for the fully discrete LDG $\theta$-scheme for parabolic singularly perturbed problems.

\section{The LDG method}
\label{sec:scheme}

Consider a two-dimensional singularly perturbed reaction--diffusion problem, expressed as
\begin{subequations}\label{spp:R-D}
	\begin{align}
	-\varepsilon \Delta u + b(x,y) u
	&= f(x,y),   &\textrm{in}&\;  \Omega =(0,1)^2,
	\\
	u &= 0,   &\textrm{on}&\; \partial \Omega,
	\end{align}
\end{subequations}
where $0<\varepsilon \ll 1$ is a perturbation parameter, and
$b(x,y) \geq 2\beta^2>0$ for any $(x,y)\in \overline{\Omega}$ 
and for some positive constant $\beta$. 
In this section, we present the LDG method for \eqref{spp:R-D}.

Let $\Omega_N=\{K_{ij}\}_{i=1,2,\dots,N_x}^{j=1,2,\dots,N_y}$
be a rectangle partition of $\Omega$ with element $K_{ij}=I_i\times J_j$,
where $I_i=(x_{i-1},x_i)$ and $J_j=(y_{j-1},y_j)$.
We set $h_{x,i}=x_i-x_{i-1}$, $h_{y,j}=y_j-y_{j-1}$, and $\hbar=\min_{K_{ij}\in\Omega_N} \min\{h_{x,i},h_{y,j}\}$. We let
\begin{equation}
\spc= \{v\in L^2 (\Omega)\colon
v|_{K} \in \mathcal{Q}^{k} (K),  K\in\Omega_N \},
\end{equation}
be the discontinuous finite-element space,
where $\mathcal{Q}^k(K)$ represents the space of polynomials on $K$
with a maximum degree of $k$ in each variable.
$\spc$ is contained in a broken Sobolev space, expressed as 
\begin{equation}
\mathcal{H}^1(\Omega_N)=\{v\in L^2(\Omega):
v|_{K}\in H^1(K), \ K\in\Omega_N\},
\end{equation}
whose function is allowed to have discontinuities across element interfaces.
For $v\in \mathcal{H}^1(\Omega_N)$ and $y\in J_j$, $j=1,2,\dots,N_y$,
we use $v^\pm_{i,y}=\lim_{x\to x_{i}^\pm}v(x,y)$
and $v^\pm_{x,j}=\lim_{y\to y_{j}^\pm}v(x,y)$
to express the traces evaluated from the four directions.
We denote
\begin{align*}
\jump{v}_{i,y}&= v^{+}_{i,y}-v^{-}_{i,y},\;
\mathrm{for}\;i=1,2,\dots,N_x-1,\ \;\jump{v}_{0,y}=v^{+}_{0,y},
\;\jump{v}_{N_x,y}=-v^{-}_{N_x,y},
\\
\jump{v}_{x,j}&= v^{+}_{x,j}-v^{-}_{x,j},\;
\mathrm{for}\;j=1,2,\dots,N_y-1,\;\jump{v}_{x,0}=v^{+}_{x,0},
\;\jump{v}_{x,N_y}=-v^{-}_{x,N_y},
\end{align*}
as the jumps on the vertical and horizontal edges, respectively.

Rewrite \eqref{spp:R-D} into an equivalent first-order system:
\begin{align}\label{para:rewrite:2d}
-p_x-q_y+bu=f, \quad
\varepsilon^{-1}p=  u_x,            \quad
\varepsilon^{-1}q=  u_y,            \quad  \textrm{in} \ \Omega.
\end{align}
Let $\dual{ \cdot}{\cdot}_{D}$ be the inner product in $L^2(D)$.
Then, the LDG scheme is defined as follows. Find
$\wN=(\uN,\pN,\qN)\in \spc^3\equiv \spc\times \spc\times \spc$
such that in each element $K_{ij}$, the variational forms
\begin{subequations}\label{LDG:scheme:2d}
	\begin{alignat}{1}
	&\dual{\pN}{\vphu_x}_{K_{ij}}
	-\dual{\widehat{\pN}_{i,y}}{\vphu^{-}_{i,y}}_{J_j}
	+\dual{\widehat{\pN}_{i-1,y}}{\vphu^{+}_{i-1,y}}_{J_j}
	\nonumber\\
	&+\dual{\qN}{\vphu_y}_{K_{ij}}
	-\dual{\widehat{\qN}_{x,j}}{\vphu^{-}_{x,j}}_{I_i}
	+\dual{\widehat{\qN}_{x,j-1}}{\vphu^{+}_{x,j-1}}_{I_i}
	+\dual{b \uN}{\vphu}_{K_{ij}}
	=\dual{f}{\vphu}_{K_{ij}},
	\\
	&\varepsilon^{-1}\dual{\pN}{\vphp}_{K_{ij}}+
	\dual{ \uN}{\vphp_x}_{K_{ij}}
	-\dual{\widehat{\uN}_{i,y}}{\vphp^{-}_{i,y}}_{J_j}
	+\dual{\widehat{\uN}_{i-1,y}}{\vphp^{+}_{i-1,y}}_{J_j}
	=0,
	\\
	&\varepsilon^{-1}\dual{\qN}{\vphq}_{K_{ij}}+
	\dual{ \uN}{\vphq_y}_{K_{ij}}
	-\dual{\widehat{\uN}_{x,j}}{\vphq^{-}_{x,j}}_{I_i}
	+\dual{\widehat{\uN}_{x,j-1}}{\vphq^{+}_{x,j-1}}_{I_i}
	=0,
	\end{alignat}
\end{subequations}
hold for any $\vph=(\vphu,\vphp,\vphq)\in \spc^3$,
where the ``hat" terms are numerical fluxes defined by
\begin{subequations}\label{flux:diffusion}
	\begin{align}
	\label{flux:diffusion:q}
	\widehat{\pN}_{i,y}
	&\;
	=\begin{cases}
	\pN^{+}_{i,y} + \lambda_{i,y}\jump{\uN}_{i,y},             &i=0,1,\dots,N_x-1,\\
	\pN^-_{N_x,y} - \lambda_{N_x,y}\uN^-_{N_x,y}, &i=N_x,
	\end{cases}
	\\
	\label{flux:diffusion:u}
	\widehat{\uN}_{i,y}
	&\;
	=\begin{cases}
	0,             &\hspace{1.5cm}  i=0,N_x,\\
	\uN^{-}_{i,y}, &\hspace{1.5cm}  i=1,2,\dots,N_x-1,\\
	\end{cases}
	\end{align}
\end{subequations}
for $y\in J_j$ and $j=1,2,\dots,N_y$.
Here, $\lambda_{i,y} \geq 0  (i=0,1,\cdots,N_x)$ are stabilisation parameters to be determined
later.
Analogously, for $x\in I_i$ and $i=1,2,\dots,N_x$,
we can define $\widehat{\qN}_{x,j}$ and $\widehat{\uN}_{x,j}$
for $j=0,1,\dots,N_y$.

Write 
$\dual{w}{v}=\sum_{K_{ij}\in \Omega_N}\dual{w}{v}_{K_{ij}}$.
Then, we write the above LDG method into a compact form:
Find $\wN=(\uN,\pN,\qN)\in \spc^3$ such that
\begin{equation}\label{compact:form:2d}
\Bln{\wN}{\vph}=\dual{f}{\vphu},
\quad \forall \vph=(\vphu,\vphp,\vphq)\in \spc^3,
\end{equation}
where
\begin{align}
\label{B:def:2d}
\Bln{\wN}{\vph}=&\;
\mathcal{T}_1(\wN;\vph)+\mathcal{T}_2(\uN;\vph)
+\mathcal{T}_3(\wN;\vphu)+\mathcal{T}_4(\uN;\vphu),
\end{align}
with
\begin{align}
\Ton{\wN}{\vph}=&\;
\sq^{-1}\dual{\pN}{\vphp}+\sq^{-1}\dual{\qN}{\vphq}+\dual{b\uN}{\vphu},
\nonumber\\
\Ttw{\uN}{\vph}=&\;
\dual{ \uN}{\vphp_x}
+\sum_{j=1}^{N_y}\sum_{i=1}^{N_x-1}\dual{\uN^{-}_{i,y}}{\jump{\vphp}_{i,y}}_{J_j}
+\dual{ \uN}{\vphq_y}
+\sum_{i=1}^{N_x}\sum_{j=1}^{N_y-1}\dual{\uN^{-}_{x,j}}{\jump{\vphq}_{x,j}}_{I_i},
\nonumber\\
\Tth{\wN}{\vphu}=&\;
\dual{\pN}{\vphu_x}
+\sum_{j=1}^{N_y}\Big[
\sum_{i=0}^{N_x-1}\dual{\pN^{+}_{i,y}}{\jump{\vphu}_{i,y}}_{J_j}
-\dual{\pN^{-}_{N_x,y}}{\vphu^{-}_{N_x,y}}_{J_j}
\Big]
\nonumber\\
+&\dual{\qN}{\vphu_y}
+\sum_{i=1}^{N_x}\Big[
\sum_{j=0}^{N_y-1}\dual{\qN^{+}_{x,j}}{\jump{\vphu}_{x,j}}_{I_i}
-\dual{\qN^{-}_{x,N_y}}{\vphu^{-}_{x,N_y}}_{I_i}
\Big],
\nonumber\\
\Tfr{\uN}{\vphu}=&\;
\sum_{j=1}^{N_y}
\sum_{i=0}^{N_x}
\dual{\lambda_{i,y}\jump{\uN}_{i,y}}{\jump{\vphu}_{i,y}}_{J_j}
+\sum_{i=1}^{N_x}
\sum_{j=0}^{N_y}
\dual{\lambda_{x,j}\jump{\uN}_{x,j}}{\jump{\vphu}_{x,j}}_{I_i}.
\nonumber
\end{align}

\section{Layer-adapted meshes}
\label{sec:meshes}

To introduce the layer-adapted meshes,
we extract some precise information from the exact 
solution of \eqref{spp:R-D} and its derivatives \cite{Han1990,Clavero2005}.
\begin{proposition}\label{thm:reg:2d}
	Assume that the solution $u$ of \eqref{spp:R-D} can be decomposed by
	\begin{equation}
	u = S+ \sum_{k=1}^4 W_{k}+\sum_{k=1}^4 Z_{k}, 
	\quad (x,y)\in \overline{\Omega},
	\end{equation}
	where S is a smooth part, $W_k$ is a boundary layer part
	and $Z_k$ is a corner layer part.
	More precisely, for $0\leq i,j\le k+2$, there exists a constant $C$
	independent of $\varepsilon$ such that
	\begin{align}\label{reg:u:2d}
	\left| \dx^{i}\dy^{j} S\right| \leq C, \quad
	\left| \dx^{i}\dy^{j}W_{1}\right|
	\leq  C\sq^{-\frac{i}{2}} e^{-\frac{\beta x}{\ssq}}, \quad
	\left| \dx^{i}\dy^{j}Z_{1}\right|
	\leq  C\sq^{-\frac{i+j}{2}}e^{-\frac{\beta(x+y)}{\ssq}},
	\end{align}
	and so on for the remaining terms.
	Here, $\dx^i\dy^j v = \frac{\partial^{i+j} v}{\partial x^i \partial y^j}$.
\end{proposition}

The layer-adapted mesh is constructed as follows:
For notational simplification,
we assume that $N_x=N_y=N$.
Let $N\geq 4$ be a multiple of four. We introduce the mesh points
\[
0=x_0<x_1<\cdots<x_{N-1}<x_{N}=1,\quad
0=y_0<y_1<\cdots<y_{N-1}<y_{N}=1,
\]
and consider a tensor-product mesh with mesh points $(x_i,y_j)$.
Because both meshes have the same structure, we only describe the mesh in the $x$-direction.

Suppose $\varphi$ is a function defined in $[0,1/4]$ with 
\begin{align} \label{phi:assumption}
\varphi(0)=0,\varphi^{\prime}>0, \varphi^{\prime\prime}\geq0.
\end{align}
We define the transition parameter
\begin{equation}\label{condition:parameters}
\tau = \min\Big\{\frac14, \frac{\sigma\ssq}{\beta}\varphi(1/4)\Big\},
\end{equation}
where $\sigma>0$ is determined later.
Assume that $\sqrt{\varepsilon}\leq N^{-1}$
means that we are in the singularly perturbed case.
Moreover, $\varepsilon$ is sufficiently small that
\eqref{condition:parameters} is replaced by 
$\tau =\frac{\sigma\ssq}{\beta}\varphi(1/4)$.

The mesh in the $x$-direction is equidistant on $[\tau,1-\tau]$ with $N/2$ elements, but it is gradually divided on $[0,\tau]$ and $[1-\tau,1]$ with $N/4$ elements.
Hence, we set the mesh points as
\begin{equation}\label{layer-adapted}
x_i=
\begin{dcases}
\frac{\sigma\ssq}{\beta}\varphi(i/N),
&\textrm{for} \; i = 0,1,\cdots,N/4,
\\
\tau+2(1-2\tau)(i/N-1/4),
&\textrm{for}\; i = N/4+1,\cdots,3N/4-1,
\\
1-\frac{\sigma\ssq}{\beta}\varphi(1-i/N),
& \textrm{for}\; i = 3N/4,\cdots,N.
\end{dcases}
\end{equation}

In Table \ref{table:functions}, we list three typical layer-adapted meshes \cite{Linss2003book}: 
Shishkin (S-mesh), Bakhvalov--Shishkin (BS mesh)
and Bakhvalov-type (B-type mesh), together with $\psi=e^{-\varphi}$ and the important quantity $\max|\psi^{\prime}|$,
which arises in error estimates.
Figure \ref{fig:division} illustrates the divisions of $\Omega$ 
and the generated meshes
for $\varepsilon=10^{-2}$ and $N=16$.

Note that for these meshes and under the previous assumption, we always have
$\tau \geq \frac{\sigma\ssq}{\beta}\ln N$.

\begin{table}[ht]
	\caption{Layer-adapted meshes}
	\label{table:functions}
	\footnotesize
	\centering
	\begin{tabular}{cccc}
		\toprule
		& S-mesh
		& BS-mesh
		& B-type mesh
		\\
		\midrule
		$\varphi(t)$             & $4t\ln N$    & $-\ln\big[1-4(1- N^{-1})t\big]$   &$-\ln\big[1-4(1-\ssq)t\big]$ \\
		$\psi(t)$ & $N^{-4t}$    & $1-4(1-N^{-1})t$   &$1-4(1-\ssq)t$  \\
		$\max|\ppri|$ & $C\ln N$       & $C$ & $C$ \\
		\bottomrule
	\end{tabular}
\end{table}

\begin{figure}  
	\centering
	\includegraphics[width=0.4\linewidth]{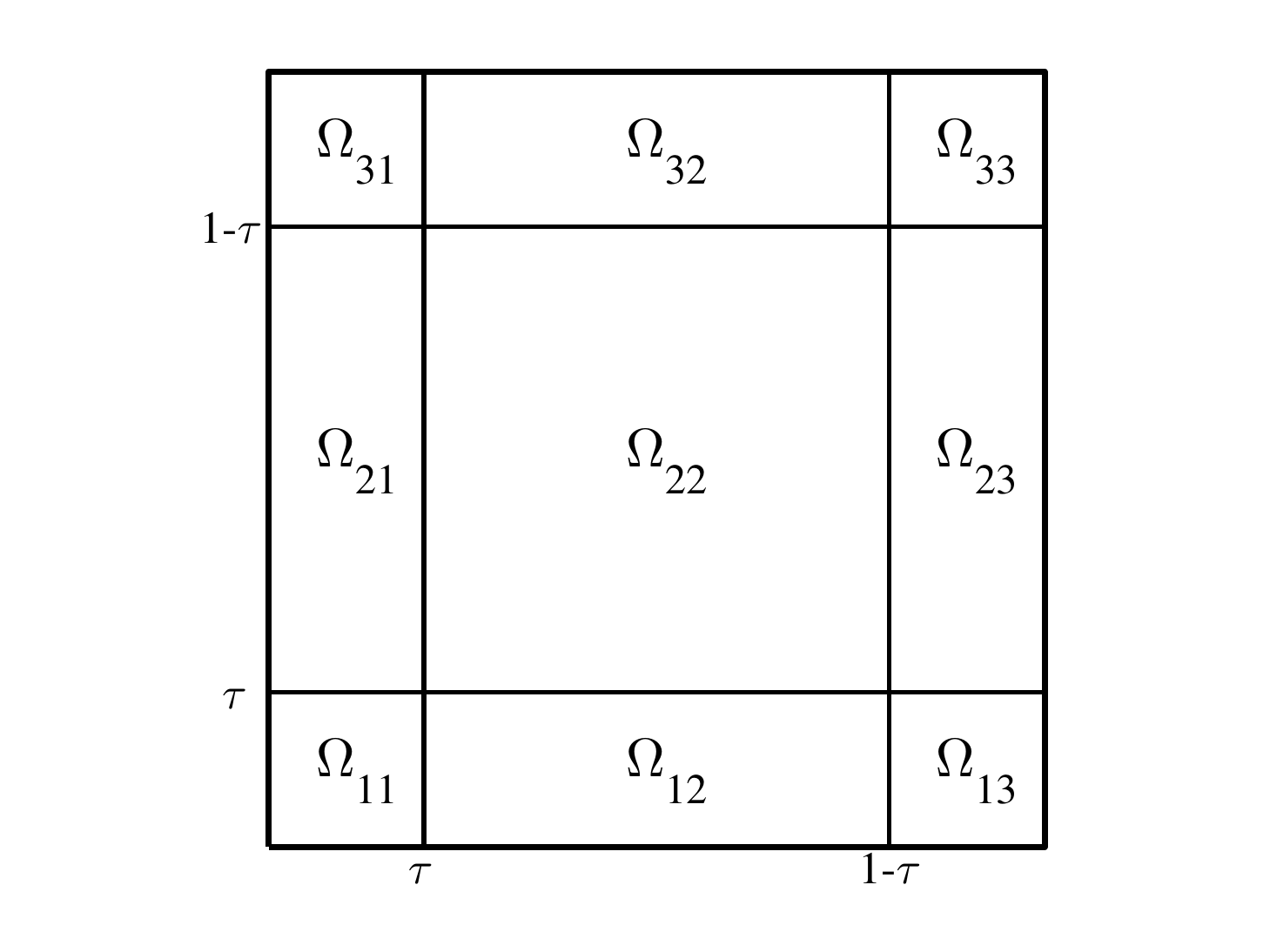}
	\includegraphics[width=0.4\linewidth]{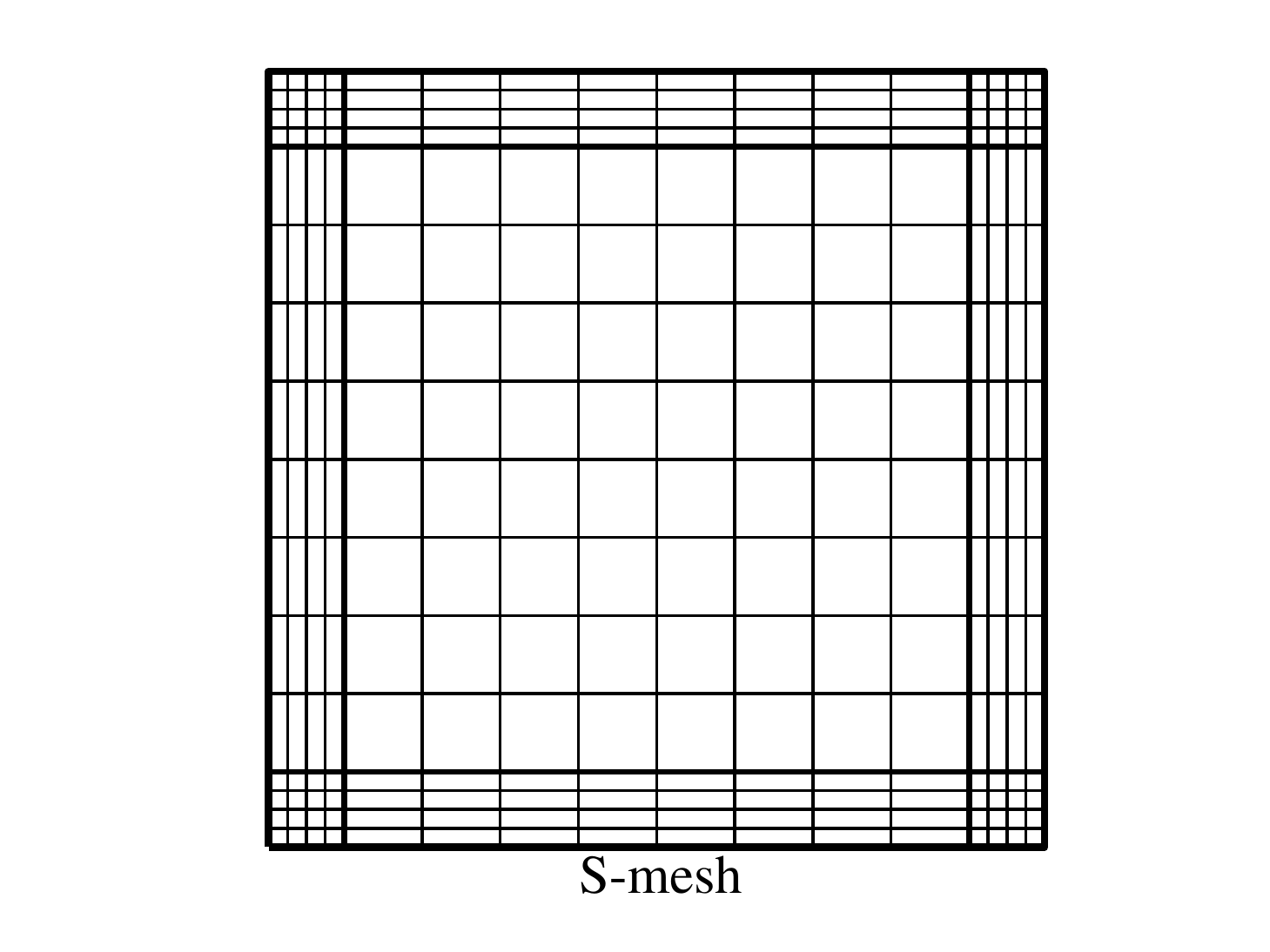}
	\\
	\includegraphics[width=0.4\linewidth]{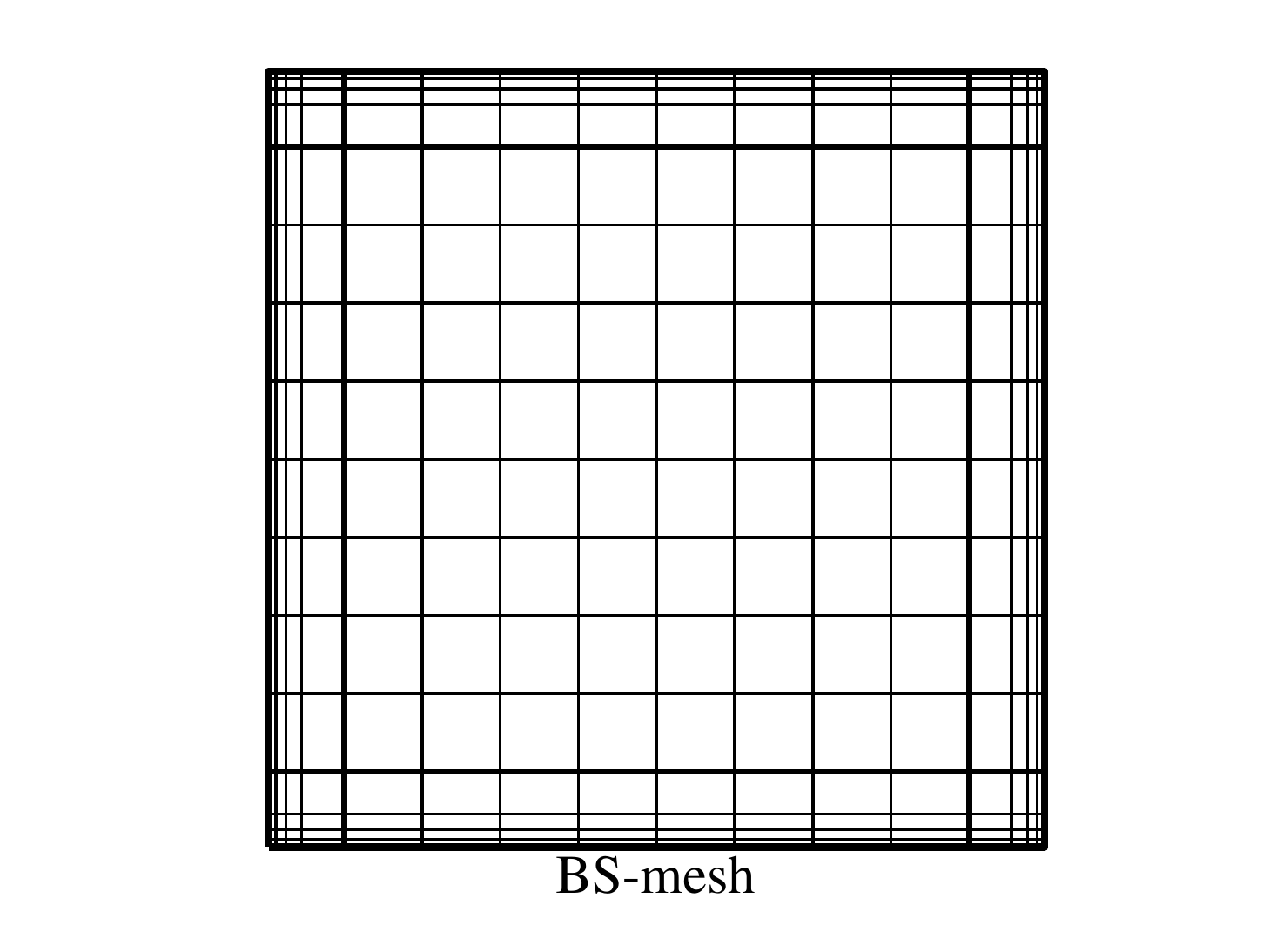}
	\includegraphics[width=0.4\linewidth]{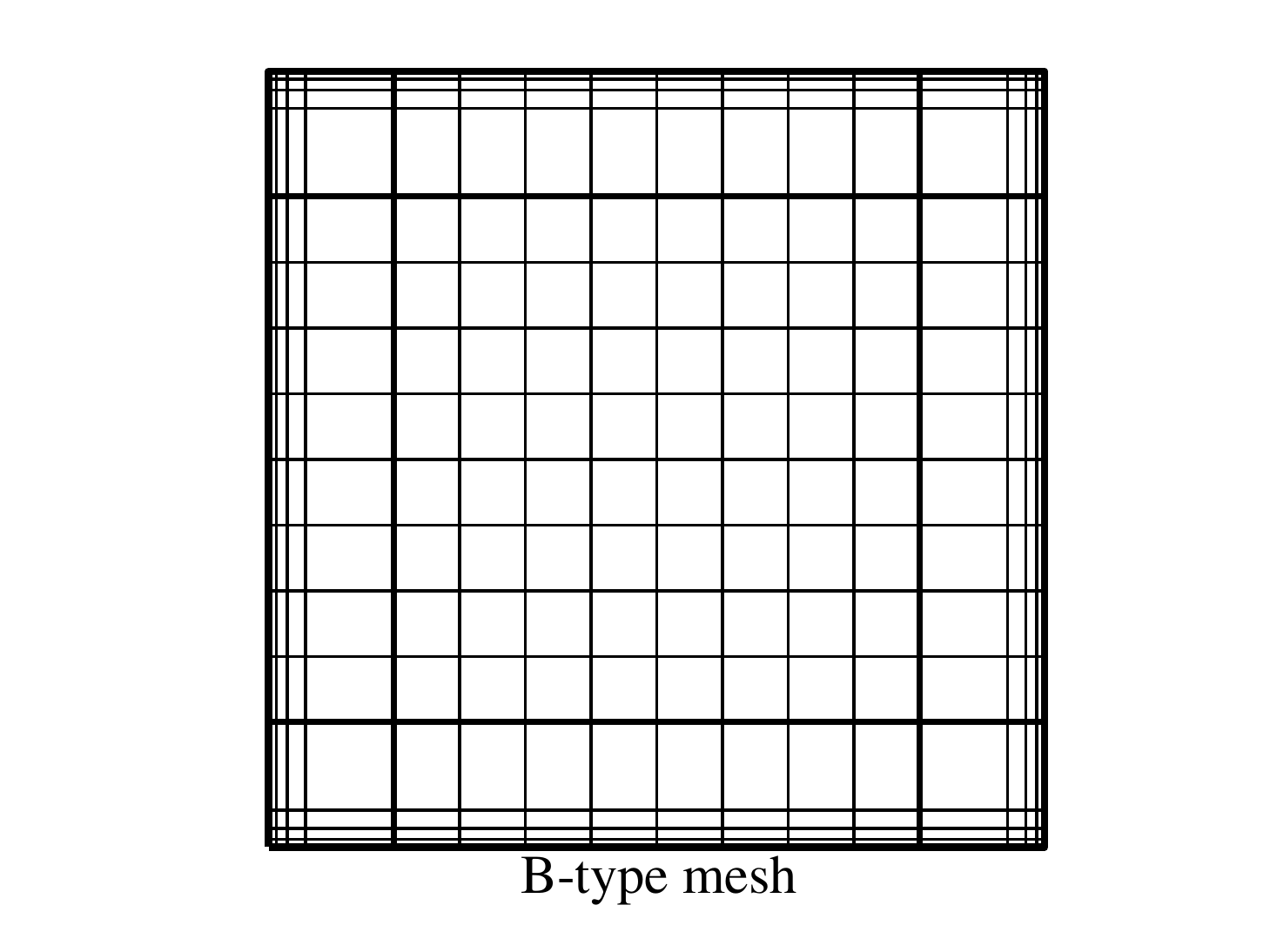}
	\caption{The division of $\Omega$}
	\label{fig:division}
\end{figure}

In the following,
we state two preliminary lemmas
that will be frequently employed in the subsequent analysis.
Because $h_{x,i}=h_{y,i}$, $i=1,2,\dots,N$,
we simply use $h_i$ to denote one of them.

\begin{lemma} \label{lemma:1}
	Suppose that
	\[
	\mcs_i =
	\min\Big\{\frac{h_i}{\ssq},1\Big\}
	e^{-\frac{\beta x_{i-1}}{\sigma\ssq}},\quad i = 1,2,\cdots,\ofh.
	\]
	Then, there exists a constant $C>0$ independent of $\sq$ and $N$ such that
	\begin{subequations} \label{bound:key}
		\begin{align} 
		\label{bound:theta:max}
		\max_{1\leq i\leq \ofh}\mcs_i
		\leq&\; CN^{-1}\max|\ppri|,
		\\
		\label{bound:theta:sum}
		\sum_{i=1}^{\ofh}\mcs_i \leq&\; C.
		\end{align}
	\end{subequations}
	
\end{lemma}

\begin{proof}
	See \cite{Cheng2020} for details.
\end{proof}

\begin{lemma} \label{lem:h:min:max}
	We have $h_{N/4+1} = h_{N/4+2} = \cdots = h_{3N/4}$ 
	and $\hbar \geq C\ssq N^{-1}\max|\ppri|$.
	For the S-type meshes,
	\begin{align}
	\label{mesh:property:S}
	\begin{cases} 
	1 \geq h_i/h_{i+1} \geq C
	& 
	i = 1,2,\cdots,\ofh-1,
	\\
	1 \geq h_{i+1}/h_i \geq C
	& 
	i = \tfh+1,\tfh+2,\cdots,N-1.
	\end{cases}
	\end{align}
	For the B-type mesh,
	\begin{align}
	\label{mesh:property:B}
	\begin{cases}
	1 \geq h_i/h_{i+1} \geq C
	& 
	i = 1,2,\cdots,\ofh-2,
	\\
	1 \geq h_{i+1}/h_i \geq C
	& 
	i = \tfh+2,\tfh+3,\cdots,N-1.
	\end{cases}
	\end{align}
	Moreover, 
	\begin{equation} \label{def:qstar}
	h_{ly} \equiv \max_{\substack{i=1,\cdots,N/4 \\i=3N/4,\cdots,N}} h_i 
	\leq C \parep
	\equiv C
	\begin{cases}
	\ssq, & \textrm{for \; S-type meshes},\\
	N^{-1}, & \textrm{for \; a\; B-type mesh}.
	\end{cases}
	\end{equation}
\end{lemma}

\begin{proof}
	It can be seen that $h_{N/4+1}=h_{N/4+2}=\dots=h_{3N/4}$.
	\eqref{mesh:property:S} and \eqref{mesh:property:B} can be verified 
	trivially (see \cite{Cheng2020}).
	By \eqref{layer-adapted}, \eqref{phi:assumption}, and $\psi=e^{-\varphi}$, we have
	\begin{align*}
	\hbar \geq&\; C\ssq N^{-1}\min|\varphi^{\prime}|
	=C\ssq N^{-1}|\varphi^{\prime}(0)|
	=C\ssq N^{-1}|\psi^{\prime}(0)|
	=C\ssq N^{-1}\max|\psi^{\prime}|.
	\end{align*}
	Now, we prove \eqref{def:qstar}. Combining the fact that $h_{ly} \leq CN^{-1}$ with the inequality
	\begin{align*}
	h_{ly} 
	\leq C \ssq N^{-1} \max|\varphi^{\prime}|
	\leq  C\ssq N^{-1} |\varphi^{\prime}(1/4)|,
	\end{align*}
	we obtain 
	\begin{align}
	h_{ly} \leq C\min\Big\{\ssq N^{-1} |\varphi^{\prime}(1/4)|,N^{-1}\Big\}, 
	\end{align}
	which implies \eqref{def:qstar}.
\end{proof}

\section{Convergence analysis}
\label{sec:convergence}

In this section, we perform
uniform convergence analysis
for the LDG method
on layer-adapted meshes.
Two related norms are considered.

The first is the energy norm, which is naturally derived from
the formulation of the LDG method; that is, 
$\enorm{\wN}^2\equiv \Bln{\wN}{\wN}$. Hence, we obtain
\begin{align}
\label{energy:norm:2d}
\enorm{\wN}^2  &\;= 
\sq^{-1}\norm{\pN}^2+\sq^{-1}\norm{\qN}^2
+\norm{b^{1/2}\uN}^2
\nonumber\\
&\;+
\sum_{j=1}^{N}\sum_{i=0}^{N}
\dual{\lambda_{i,y}}{\jump{\uN}^2_{i,y}}_{J_j}
+\sum_{i=1}^{N}\sum_{j=0}^{N}
\dual{\lambda_{x,j}}{\jump{\uN}^2_{x,j}}_{I_i},
\end{align}
by using integration by parts and some trivial manipulations.
Here,\\
$\norm{z}^2 = \sum_{K\in\Omega_{N}} \norm{z}^2_{K}$ and $\norm{z}^2_{K} = \dual{z}{z}_{K}$.

However,  
this norm is inadequate for reaction--diffusion problems
because the layer contributions are not ``seen" In fact,
letting $u=e^{-\beta (x+y)/\ssq}$, we have
$\enorm{(u,\sq u_x,\sq u_y)}=\mathcal{O}(\sfq)$,
which vanishes as $\varepsilon\rightarrow 0$.
Thus, the following ``balanced" norm is introduced:
\begin{align}
\label{balanced:norm:2d}
\ba{\wN}^2 &\;=\;
\sq^{-3/2} \norm{\pN}^2+\sq^{-3/2}\norm{\qN}^2
+\norm{b^{1/2}\uN}^2
\nonumber\\
&\;+
\sum_{j=1}^{N}\sum_{i=0}^{N}
\dual{1}{\jump{\uN}^2_{i,y}}_{J_j}
+\sum_{i=1}^{N}\sum_{j=0}^{N}
\dual{1}{\jump{\uN}^2_{x,j}}_{I_i}.
\end{align}

In the following subsections, 
we perform convergence analysis for these two norms. 
Different projections are introduced, 
and the related approximation properties 
are investigated.

\subsection{Convergence of balanced norm}
First, we analyse the LDG method for the balanced norm \eqref{balanced:norm:2d}.
Let $\omega\in C^1(\Omega_N)$ and
$\omega\geq\omega_0>0$
be a general weight function.
We define the piecewise local weight $L^2$ projection $\Pi_\omega$ as follows:
For each element $K\in \Omega_N$ and for any $z\in
L^2(\Omega_N)$, $\Pi_{\omega} z \in V_N$ satisfies
\begin{align} \label{def:prj:WL2}
\dual{\omega\Pi_{\omega} z}{\vphu}_{K}
= \dual{\omega z}{\vphu}_{K},
\quad \forall \vphu\in \mathcal{Q}^{k}(K).
\end{align}   
In the special case of $\omega=1$, 
this operator reduces to the classical
local $L^2$ projection, which is denoted by $\Pi$.

\begin{lemma}\cite{Apel1999}
	There exists a constant $C>0$, independent of the element size and $z$, such that
	\begin{subequations}\label{2L2:stb:app}
		\begin{align}
		\label{2L2:stb}
		\norm{\Pi_{\omega} z}_{L^m(\kij)}\leq &\; C
		\norm{z}_{L^m(\kij)},
		\\
		\label{2L2:app}
		\norm{z-\Pi_{\omega} z}_{L^m(\kij)}\leq &\;
		C \Big[h_{i}^{k+1}\norm{\partial_x^{k+1}z}_{L^m(\kij)}
		+h_{j}^{k+1}\norm{\partial_y^{k+1}z}_{L^m(\kij)}
		\Big],
		\end{align}
	\end{subequations}
	where $m \in \{2,\infty\}$.
\end{lemma}

\begin{lemma}\label{lemma:L2}
	Let $\sigma\geq k+1.5$. Then, there holds
	\begin{subequations}\label{2dL2:property}
		\begin{align}
		\label{app:u:L2}
		\norm{u-\Pi_{\omega} u}
		&\; \leq C 
		\Big[ \sfq (N^{-1}\max|\ppri|) ^{k+1}+  N^{-(k+1)} \Big],
		\\
		\label{app:u:inf}
		\norm{u-\Pi_{\omega} u}_{L^\infty(\Omega_N)}
		&\; \leq C 
		(N^{-1}\max|\ppri|) ^{k+1},
		\\
		\label{app:u:bd:L2}
		\sum_{j=1}^{N}\sum_{i=1}^{N-1}
		\norm{(u-\Pi_{\omega} u)_{i,y}^{-}}^2_{J_j}
		&\;\leq C (N^{-1}\max|\psi^{\prime}|)^{2k+1},
		\\
		\label{app:u:N:L2}
		\sum_{j=1}^{N}\sum_{i=0}^{N}
		\dual{1}{\jump{u-\Pi_{\omega} u}^2_{i,y}}_{J_j}
		&\; \leq C(N^{-1}\max|\psi^{\prime}|)^{2k+1},
		\\
		\label{app:pq:L2}
		\norm{p-\Pi_{\omega} p}
		&\; \leq C  \sq^{\frac34} (N^{-1}\max|\ppri|)^{k+1},
		\\
		\label{app:pq:N:L2}
		\sum_{j=1}^{N} 
		\norm{(p-\Pi_{\omega} p)_{N,y}^{-}}^2_{J_j}
		&\;\leq C \sq (N^{-1}\max|\psi^{\prime}|)^{2(k+1)},
		\\
		\label{app:pq:bd:L2}
		\sum_{j=1}^{N}\sum_{i=0}^{N-1}
		\norm{(p-\Pi_{\omega} p)_{i,y}^{+}}^2_{J_j}
		&\; \leq C\sq(N^{-1}\max|\psi^{\prime}|)^{2k+1},
		\end{align}
	\end{subequations}
	where $C>0$ is independent of $\varepsilon$ and $N$.
	A similar procedure applies for $u$ and $q$ in other spatial directions. 
\end{lemma}
\begin{proof}
	Let $\eta_u=u-\Pi_{\omega} u$, $\eta_p=p-\Pi_{\omega} p$, and $\eta_q=q-\Pi_{\omega} q$.
	We prove \eqref{2dL2:property} separately.
	
	(1) Prove \eqref{app:u:L2}.
	Recalling the decomposition $u = S+ \sum_{k=1}^4 W_{k}+\sum_{k=1}^4 Z_{k}$,
	we have
	\begin{align}
	\label{eta:L2norm}
	\norm{\eta_u}\leq \norm{\eta_{S}}
	+ \sum_{k=1}^4 \norm{\eta_{W_k}}
	+ \sum_{k=1}^4 \norm{\eta_{Z_k}}.
	\end{align}
	
	Using \eqref{2L2:app} with $m=2$
	and \eqref{reg:u:2d}, we obtain 
	\begin{align}\label{eta:S}
	\norm{\eta_{S}}\leq C N^{-(k+1)}
	\Big[ 
	\norm{\dx^{k+1}S}+\norm{\dy^{k+1}S}
	\Big]
	\leq C N^{-(k+1)},
	\end{align}
	because $h_i \leq C N^{-1}$ for $i=1,2,\dots,N$.
	
	Moreover, denote $\xl=\Omega_{11}\cup\Omega_{21}\cup\Omega_{31}$,
	$\xm=\Omega_{12}\cup\Omega_{22}\cup\Omega_{32}$, and $\xr=\Omega_{13}\cup\Omega_{23}\cup\Omega_{33}$; then, we have
	\begin{equation}\label{w1}
	\norm{\eta_{W_1}}^2
	=\sum_{\kij \in \xl}\norm{\eta_{W_1}}_{\kij}^2
	+\sum_{\kij \in \xm\cup\xr}
	\norm{\eta_{W_1}}_{\kij}^2
	\equiv \Lambda_1 + \Lambda_2.
	\end{equation}
	
	For $\Lambda_1$,
	we use \eqref{2L2:stb}
	and \eqref{2L2:app} with $m=2$ to obtain the two estimates
	\begin{align}\label{lambda:stb}
	\Lambda_1\leq &\; C\sum_{\kij \in \xl}
	\norm{W_1}_{\kij}^2
	\leq  C\sum_{\kij \in \xl}\norm{e^{-\frac{\beta x}{\ssq}}}_{\kij}^2,
	\\
	\label{lambda:app}
	\Lambda_1\leq &\; C\sum_{\kij \in \xl}
	\Big[
	h_i^{2(k+1)} \norm{\dx^{k+1}W_1}_{\kij}^2+
	h_j^{2(k+1)} \norm{\dy^{k+1}W_1}_{\kij}^2
	\Big]
	\nonumber\\
	\leq &\; C \sum_{\kij \in \xl}
	\Big[
	\Big(\frac{h_i}{\ssq}\Big)^{2(k+1)}+h_j^{2(k+1)}
	\Big]
	\norm{e^{-\frac{\beta x}{\ssq}}}_{\kij}^2,
	\end{align}
	respectively, where we have used \eqref{reg:u:2d}.
	Combining \eqref{lambda:stb} with \eqref{lambda:app} 
	and using $h_{i}/\ssq \geq C N^{-1}\geq C h_{j}$,
	$\sum_{j=1}^N h_j=1$, $\sigma\geq k+1.5$, and \eqref{bound:key},
	we have
	\begin{align}\label{lambda:1}
	\Lambda_1
	\leq &\;
	C\sum_{\kij \in \xl}
	\min\Big\{\Big(\frac{h_i}{\ssq}\Big)^{2(k+1)}
	+h_j^{2(k+1)},1\Big\}
	\norm{e^{-\frac{\beta x}{\ssq}}}_{\kij}^2
	\nonumber\\
	\leq &\; C\sum_{i=1}^{\ofh}\min\Big\{\Big(\frac{h_i}{\ssq}\Big)^{2(k+1)},1\Big\} 
	\ssq e^{-\frac{2\beta x_{i-1}}{\ssq}}
	\min\Big\{\frac{h_i}{\ssq},1\Big\} \Big(\sum_{j=1}^{N} h_j\Big)
	\nonumber\\
	\leq &\;
	C\ssq \max_{1\leq i\leq\ofh}\mcs_i^{2(k+1)}
	\sum_{i=1}^{\frac{N}{4}}\mcs_i
	\leq C\ssq(N^{-1}\max|\psi^\prime|)^{2(k+1)},
	\end{align}
	where we used the trivial inequality
	\[
	\norm{ e^{-\frac{\beta x}{\ssq}} }_{\kij}^2
	=\frac{\ssq}{2\beta}h_j e^{-\frac{2\beta x_{i-1}}{\ssq}}
	(1-e^{-\frac{2\beta h_i}{\ssq}})
	\leq C \ssq h_j e^{-\frac{2\beta x_{i-1}}{\ssq}}
	\min\Big\{\frac{h_i}{\ssq},1\Big\},
	\]
	because $1-e^{-x}\leq \min\{1,x\}$ for $x\geq0$.
	For $\Lambda_2$, we use \eqref{2L2:stb} and \eqref{reg:u:2d} to obtain 
	\begin{align} \label{lambda:2}
	\Lambda_2 \leq 
	C\sum_{\kij \in \xm\cup\xr}
	\norm{W_1}_{\kij}^2 
	\leq
	C\int_{0}^{1}dy\int_{\tau}^{1} e^{-\frac{2\beta x}{\ssq} } dx
	\leq  C\ssq e^{-\frac{2\beta \tau}{\ssq}}
	\leq C\ssq N^{-2\sigma}.
	\end{align}
	Inserting \eqref{lambda:1} and \eqref{lambda:2} into \eqref{w1} yields
	\begin{align}\label{eta:L2:W1}
	\norm{\eta_{W_1}}
	\leq C\sfq[(N^{-1}\max|\psi^{\prime}|)^{k+1}+N^{-\sigma}].
	\end{align}
	
	In addition,
	\[
	\norm{\eta_{Z_1}}^2
	=\sum_{\kij \in \Omega_{11}} \norm{\eta_{Z_1}}_{\kij}^2
	+\sum_{\kij \in \Omega_N\setminus\Omega_{11}} \norm{\eta_{Z_1}}_{\kij}^2
	\equiv \Xi_1 + \Xi_2.
	\]
	Using \eqref{2L2:stb}
	and \eqref{2L2:app} with $m=2$ gives the two estimates
	\begin{align}\label{xi:stb}
	\Xi_1\leq &\; C\sum_{\kij \in \Omega_{11}}
	\norm{Z_1}_{\kij}^2
	\leq  C\sum_{\kij \in \Omega_{11}}\norm{e^{-\frac{\beta (x+y)}{\ssq}}}_{\kij}^2,
	\\
	\label{xi:app}
	\Xi_1\leq &\; C\sum_{\kij \in \Omega_{11}}
	\Big[
	h_i^{2(k+1)} \norm{\dx^{k+1}Z_1}_{\kij}^2+
	h_j^{2(k+1)} \norm{\dy^{k+1}Z_1}_{\kij}^2
	\Big]
	\nonumber\\
	\leq &\; C \sum_{\kij \in \Omega_{11}}
	\Big[
	\Big(\frac{h_i}{\ssq}\Big)^{2(k+1)}
	+\Big(\frac{h_j}{\ssq}\Big)^{2(k+1)}
	\Big]
	\norm{e^{-\frac{\beta (x+y)}{\ssq}}}_{\kij}^2,
	\end{align}
	respectively. This leads to
	\begin{align}\label{xi:1}
	\Xi_1
	\leq &\;
	C\sum_{\kij \in \Omega_{11}}
	\min\Big\{\Big(\frac{h_i}{\ssq}\Big)^{2(k+1)}
	+\Big(\frac{h_j}{\ssq}\Big)^{2(k+1)},1\Big\}
	\norm{e^{-\frac{\beta (x+y)}{\ssq}}}_{\kij}^2
	\nonumber\\
	\leq &\; C\sum_{i=1}^{\ofh}\min\Big\{\Big(\frac{h_i}{\ssq}\Big)^{2(k+1)},1\Big\} 
	\sq e^{-\frac{2\beta x_{i-1}}{\ssq}}
	\min\Big\{\frac{h_i}{\ssq},1\Big\} \sum_{j=1}^{\ofh} 
	e^{-\frac{2\beta y_{j-1}}{\ssq}}
	\min\Big\{\frac{h_j}{\ssq},1\Big\}
	\nonumber\\
	&\; + C\sum_{j=1}^{\ofh}\min\Big\{\Big(\frac{h_j}{\ssq}\Big)^{2(k+1)},1\Big\} 
	\sq e^{-\frac{2\beta y_{j-1}}{\ssq}}
	\min\Big\{\frac{h_j}{\ssq},1\Big\} \sum_{i=1}^{\ofh} 
	e^{-\frac{2\beta x_{i-1}}{\ssq}}
	\min\Big\{\frac{h_i}{\ssq},1\Big\}
	\nonumber\\
	\leq &\;
	C\sq \max_{1\leq i\leq\ofh}
	\Theta_i^{2(k+1)}
	\sum_{i=1}^{\frac{N}{4}}\Theta_i
	\sum_{j=1}^{\frac{N}{4}}\Theta_j
	+ 
	C\sq \max_{1\leq j\leq\ofh}
	\Theta_j^{2(k+1)}
	\sum_{j=1}^{\frac{N}{4}}	\Theta_j
	\sum_{i=1}^{\frac{N}{4}} \Theta_i
	\nonumber\\
	\leq &\;
	C\sq(N^{-1}\max|\psi^\prime|)^{2(k+1)},
	\end{align}
	where we used $\sigma \geq k+1.5$, Lemma \ref{lemma:1},
	the trivial inequality $\min\{1,a+b\}\leq\min\{1,a\}+\min\{1,b\}$, 
	and 
	\begin{align*}
	\norm{ e^{-\frac{\beta (x+y)}{\ssq}} }_{\kij}^2
	=&\;\frac{\sq}{4\beta^2} 
	e^{-\frac{2\beta (x_{i-1}+y_{j-1})}{\ssq}}
	(1-e^{-\frac{2\beta h_i}{\ssq}})
	(1-e^{-\frac{2\beta h_j}{\ssq}})
	\nonumber\\
	\leq&\; C \sq
	e^{-\frac{2\beta (x_{i-1}+y_{j-1})}{\ssq}}
	\min\Big\{\frac{h_i}{\ssq},1\Big\}
	\min\Big\{\frac{h_j}{\ssq},1\Big\}.
	\end{align*}

	Similar to before, we have
	\begin{align} \label{xi:2}
	\Xi_2 \leq 
	C\sum_{\kij \in \Omega_N\setminus\Omega_{11}}
	\norm{Z_1}_{\kij} ^2
	\leq C\int_{\Omega_N\setminus\Omega_{11}} e^{-\frac{2\beta (x+y)}{\ssq}} dxdy
	\leq  C\sq e^{-\frac{2\beta \tau}{\ssq}}
	\leq C\sq N^{-2\sigma}.
	\end{align}
	
	From \eqref{xi:1} and \eqref{xi:2}, we obtain
	\begin{align}\label{eta:L2:Z1}
	\norm{\eta_{Z_1}}
	\leq C \ssq (N^{-1}\max|\ppri|)^{k+1}.
	\end{align}
	Similarly, we can bound the other terms in \eqref{eta:L2norm}
	and arrive at \eqref{app:u:L2}.
	
	(2) Prove \eqref{app:u:inf} and \eqref{app:u:bd:L2}.	
	It can be seen that
	\begin{align*}
	\sum_{j=1}^{N}\sum_{i=1}^{N-1}
	\norm{(\eta_S)_{i,y}^{-}}_{J_j}^2
	\leq&\; \sum_{j=1}^{N}\sum_{i=1}^{N-1} h_j \norm{\eta_S}_{L^\infty(\kij)}^2
	\leq C \sum_{j=1}^{N}\sum_{i=1}^{N-1} h_jN^{-2(k+1)}
	\leq C N^{-(2k+1)}.
	\end{align*}
	
	For $\kij \in \xl$, we obtain from the $L^\infty$-stability \eqref{2L2:stb}, the $L^\infty$-approximation \eqref{2L2:app},
	and Lemma \ref{lemma:1} that
	\begin{align}\label{w1:0}
	&\norm{\eta_{W_1}}_{L^\infty(\kij)}^2 \nonumber\\
	\leq&\; C \min\Big\{
	\norm{W_1}_{L^\infty(\kij)}^2,
	h_i^{2(k+1)} \norm{\dx^{k+1}W_1}_{L^\infty(\kij)}^2+
	h_j^{2(k+1)} \norm{\dy^{k+1}W_1}_{L^\infty(\kij)}^2\Big\}
	\nonumber\\
	\leq&\;
	C\min\Big\{1,\Big(\frac{h_i}{\ssq}\Big)^{2(k+1)}
	+h_j^{2(k+1)}\Big\} 
	e^{-\frac{2\beta x_{i-1}}{\ssq}}
	\nonumber\\
	\leq&\;
	\leq C \Theta_i^{2(k+1)}
	\leq C (N^{-1}\max|\ppri|)^{2(k+1)},
	\end{align}
	where we have used \eqref{reg:u:2d} and $h_i/\ssq \geq CN^{-1} \geq h_j$. 
	For $\kij \in \xm\cup\xr$, we obtain from the $L^\infty$-stability \eqref{2L2:stb} and $\sigma \geq k+1$ that
	\begin{align*}
	\norm{\eta_{W_1}}_{L^\infty(\kij)}^2 \leq &\;
	C \norm{W_1}_{L^\infty(\kij)}^2
	\leq  C e^{-\frac{2\beta x_i}{\ssq}} 
	\leq  C e^{-\frac{2\beta \tau}{\ssq}}
	\leq C N^{-2(k+1)}.
	\end{align*}
	Consequently, we obtain from \eqref{bound:key} that
	\begin{align}
	\sum_{j=1}^{N}\sum_{i=1}^{N-1}\norm{(\eta_{W_1})_{i,y}^{-}}_{J_j}^2
	\leq&\; C\sum_{i=1}^{N/4}\sum_{j=1}^{N} h_j \Theta_i^{2(k+1)}
	+C\sum_{i=N/4+1}^{N}\sum_{j=1}^{N} h_j N^{-2(k+1)}\nonumber
	\\
	\leq&\; C\max_{1\leq i\leq N/4} \Theta_i^{2k+1}
	\sum_{i=1}^{N/4} \Theta_i \sum_{j=1}^{N} h_j
	+C N^{-(2k+1)}\nonumber
	\\
	\leq&\; C[(N^{-1}\max|\ppri|)^{(2k+1)}+N^{-(2k+1)}].
	\end{align}

	Moreover, we have $\norm{\eta_{Z_1}}_{L^\infty(\kij)}^2\leq C N^{-2(k+1)}$	 
	for $\kij \in \Omega_N\setminus\Omega_{11}$. 
	For $\kij \in \Omega_{11}$, we obtain
	\begin{align*}
	&\norm{\eta_{Z_1}}_{L^\infty(\kij)}^2 \nonumber\\
	\leq&\; C \min\Big\{
	\norm{Z_1}_{L^\infty(\kij)}^2,
	h_i^{2(k+1)} \norm{\dx^{k+1}Z_1}_{L^\infty(\kij)}^2+
	h_j^{2(k+1)} \norm{\dy^{k+1}Z_1}_{L^\infty(\kij)}^2\Big\}
	\nonumber\\
	\leq&\;
	C\min\Big\{1,\Big(\frac{h_i}{\ssq}\Big)^{2(k+1)}
	+\Big(\frac{h_j}{\ssq}\Big)^{2(k+1)}\Big\} 
	e^{-\frac{2\beta (x_{i-1}+y_{j-1})}{\ssq}}
	\nonumber\\
	\leq&\;  
	C (\Theta_i^{2(k+1)} + \Theta_j^{2(k+1)})
	\leq C (N^{-1}\max|\ppri|)^{2(k+1)},
	\end{align*}
	using $\sigma\geq k+1$. In a similar fashion, we obtain
	\begin{align}
	\sum_{j=1}^{N}\sum_{i=1}^{N-1}\norm{(\eta_{Z_1})_{i,y}^{-}}_{J_j}^2
	\leq C[(N^{-1}\max|\ppri|)^{(2k+1)}+N^{-(2k+1)}].
	\end{align}
	
	From the solution decomposition and similar arguments for the other terms, we arrive at \eqref{app:u:inf} and \eqref{app:u:bd:L2}.

	(3) Prove \eqref{app:u:N:L2}.
	We start from the following inequality:
	\begin{align*}
	\sum_{j=1}^{N}\sum_{i=0}^{N}
	\dual{1}{\jump{\eta_u}^2_{i,y}}_{J_j}
	\leq&\;
	2\Big[\sum_{j=1}^{N}\sum_{i=1}^{N}\int_{J_j}[(\eta_u)^+_{i-1,y}]^2 dy
	+\sum_{j=1}^{N}\sum_{i=1}^{N}\int_{J_j}[(\eta_u)^-_{i,y}]^2 dy
	\Big].
	\end{align*}
	For the first term, we notice that
	\[
	\sum_{j=1}^{N}\sum_{i=1}^{N}\int_{J_j}[(\eta_u)^+_{i-1,y}]^2 dy
	\leq
	\sum_{j=1}^{N}\sum_{i=1}^{N}h_{j}\norm{\eta_u}^2_{L^\infty(K_{ij})}
	\]
	and proceed as before.
	For the second term, we use \eqref{app:u:bd:L2}.
	Thus, \eqref{app:u:N:L2} follows.
	
	The remaining inequalities of \eqref{2dL2:property}
	can be proved analogously; we omit the	details here.
\end{proof}

\begin{theorem} \label{thm:balanced}
	Suppose that $\lambda_{i,y}=\lambda_{x,j}= \ssq$ for  $i,j=0,1,\cdots,N$.
	Let $\bm w=(u,p,q)$ be the solution to problem \eqref{spp:R-D},
	satisfying Proposition \ref{thm:reg:2d}; furthermore, let $\wN=(\uN,\pN,\qN)\in \spc^3$ be the numerical solution of
	the LDG scheme \eqref{LDG:scheme:2d} on layer-adapted meshes \eqref{layer-adapted} when $\sigma\geq k+1.5$.
	Then, there exists a constant $C>0$, independent of $\sq$ and $N$,
	such that
	\begin{align}\label{error:balanced:general}
	\ba{\bm w-\wN} \leq	CN^{-k}(\max|\ppri|)^{k+1/2}.
	\end{align}
\end{theorem}

\begin{proof}
	Denote $\bm e  =\bm w-\wN=\bm \eta-\bm \xi$ as
	\begin{subequations}\label{error:decomposition}
		\begin{align}
		\bm \eta &\; = (\eta_u,\eta_p,\eta_q)
		= (u-\Pi_{b} u,p-\Pi p,q-\Pi q),
		\\
		\bm \xi &\; = (\xi_u,\xi_p,\xi_q)
		= ( \uN-\Pi_{b} u, \pN-\Pi p, \qN-\Pi q)\in \spc^3,
		\end{align}
	\end{subequations}
	where
	$\Pi_{b}$ is defined in \eqref{def:prj:WL2} with weight $\omega=b$.
	
	From Proposition \ref{thm:reg:2d}
	and the consistency of numerical flux,
	we obtain the error equation:
	\begin{align}
	\Bln{\bm \xi}{\vph} =&\; \Bln{\bm\eta}{\vph}
	\\
	=&\; \Ton{\bm \eta}{\vph}+\Ttw{\eta_u}{\vph}
	+\Tth{\bm \eta}{\vphu}+\Tfr{\eta_u}{\vphu}
	\nonumber.
	\end{align}
	
	It can be seen that $\Ton{\bm \eta}{\vph}=0$.	
	
	To bound $\Ttw{\eta_u}{\vph}$, we use the Cauchy--Schwarz inequality,
	inverse inequality, $C \ssq N^{-1}\max|\ppri| \leq \hbar \leq CN^{-1}$,
	\eqref{app:u:inf}, and $\sq^{-1/2}\norm{\vphp} \leq \enorm{\vph}$ to obtain
	\begin{align*}
	|\dual{\eta_u}{\vphp_x}|\leq&\; \sum_K \norm{\eta_u}_K \norm{\vphp_x}_K
	\leq C\sum_K |K|^{1/2}\norm{\eta_u}_{L^\infty(K)} h_x^{-1} \norm{\vphp}_K
	\\
	\leq&\; C\Big(\sum_K \frac{h_y}{h_x}\norm{\eta_u}^2_{L^\infty(K)}
	\Big)^{1/2} \norm{\vphp}
	\leq
	C \sfq N^{-k}(\max|\ppri|)^{k+1/2} \enorm{\vph},
	\end{align*}
	and we used \eqref{app:u:bd:L2} to obtain
	\begin{align*}
	\Big|\sum_{j=1}^{N}\sum_{i=1}^{N-1}
	\dual{(\eta_u)_{i,y}^{-}}{\jump{\vphp}_{i,y}}_{J_j}\Big|
	\leq&\;
	C\Big(
	\sum_{j=1}^{N}\sum_{i=1}^{N-1}
	\norm{(\eta_u)_{i,y}^{-}}_{J_j}^2
	\Big)^{1/2} (\hbar^{-1/2}\norm{\vphp})
	\nonumber\\
	\leq&\; C \sfq (N^{-1}\max|\ppri|)^{k} \enorm{\vph}.
	\end{align*}
	Consequently, we have
	\begin{align} \label{bnorm:term:2}
	\Ttw{\eta_u}{\vph} \leq&\;C \sfq N^{-k}(\max|\ppri|)^{k+1/2} \enorm{\vph}.
	\end{align}
	
	Analogously, we have
	\begin{align} \label{bnorm:term:3}
	\Tth{\eta}{\vphu} &\; =
	\sum_{j=1}^{N}
	\Big[
	\sum_{i=0}^{N-1}
	\dual{(\eta_p)^{+}_{i,y}}{\jump{\vphu}^{-}_{i,y}}_{J_j}
	-\dual{(\eta_p)^{-}_{N,y}}{\vphu^{-}_{N,y}}_{J_j}
	\Big]
	\nonumber\\
	&\;\quad +
	\sum_{i=1}^{N}
	\Big[
	\sum_{j=0}^{N-1}
	\dual{(\eta_q)^{+}_{x,j}}{\jump{\vphu}^{-}_{x,j}}_{I_i}
	-\dual{(\eta_q)^{-}_{x,N}}{\vphu^{-}_{x,N}}_{I_i}
	\Big]
	\nonumber\\
	&\; \leq
	C\Big(
	\sum_{j=1}^{N}\sum_{i=0}^{N-1}
	\sq^{-1}\norm{(\eta_p)_{i,y}^{+}}_{J_j}^2 
	+\sum_{j=1}^{N}\sq^{-1}\norm{(\eta_p)_{N,y}^{-}}^2_{J_j}
	\Big)^{1/2} (\ssq\hbar^{-1/2}\norm{\vphu})
	\nonumber\\
	&\;\quad+C\Big(
	\sum_{i=1}^{N}\sum_{j=0}^{N-1}
	\sq^{-1}\norm{(\eta_q)_{x,j}^{+}}_{I_i}^2
	+\sum_{i=1}^{N}\sq^{-1}\norm{(\eta_q)_{x,N}^{-}}^2_{I_i}
	\Big)^{1/2} (\ssq\hbar^{-1/2}\norm{\vphu})
	\nonumber\\
	&\; \leq C\hbar^{-1/2} \ssq
	(N^{-1}\max|\ppri|)^{k+1/2}
	\norm{\vphu}
	\nonumber\\
	&\; \leq C \sfq 
	(N^{-1}\max|\ppri|)^{k} \enorm{\vph}.
	\end{align}
	Here, \eqref{app:pq:N:L2} and \eqref{app:pq:bd:L2} were used.
	
	Finally, we use the Cauchy--Schwarz inequality, \eqref{app:u:N:L2}, and the assumption that $\lambda_{i,y}=\lambda_{x,j}=\ssq$ for $i,j=0,1,\cdots,N$, to obtain
	\begin{align}  \label{bnorm:term:4}
	\Tfr{\eta_u}{\vphu} = 
	&\;
	\sum_{j=1}^{N}
	\sum_{i=0}^{N}
	\dual{\lambda_{i,y}\jump{\eta_u}_{i,y}}
	{\jump{\vphu}_{i,y}}_{J_j}
	+\sum_{i=1}^{N}
	\sum_{j=0}^{N}
	\dual{\lambda_{x,j}\jump{\eta_u}_{x,j}}
	{\jump{\vphu}_{x,j}}_{I_i}
	\nonumber\\
	\leq &\;
	\Big[
	\sum_{i=0}^{N}\sum_{j=1}^{N}
	\lambda_{i,y} \dual{1}{\jump{\eta_u}^2_{i,y}}_{J_j}
	+\sum_{j=0}^{N}\sum_{i=1}^{N}
	\lambda_{x,j} \dual{1}{\jump{\eta_u}^2_{x,j}}_{I_i}
	\Big]^{1/2} \enorm{\vph}
	\nonumber\\
	\leq &\;
	C\sfq (N^{-1}\max|\ppri|)^{k+1/2} \enorm{\vph}.
	\end{align}
	
	From \eqref{bnorm:term:2}--\eqref{bnorm:term:4}, we have
	\begin{align}
	\enorm{\bm \xi}^2 =B(\bm \xi;\bm \xi)
	=B(\bm \eta;\bm \xi)
	\leq C \sfq N^{-k}(\max|\ppri|)^{k+1/2} \enorm{\bm \xi},
	\end{align}
	which implies
	\begin{align} \label{bnorm:xi}
	\ba{\bm \xi} \leq \sq^{-1/4}\enorm{\bm \xi} \leq 
	C \sfq N^{-k}(\max|\ppri|)^{k+1/2}.
	\end{align}
	Using \eqref{2dL2:property} and a trivial inequality, we derive \eqref{error:balanced:general}.
\end{proof}

\subsection{Improvement of convergence in energy norm}

In this subsection, we perform an elaborate analysis and
establish an optimal convergence result in the energy norm.
The following local Gauss--Radau projections are required.

For each element $K_{ij}\in \Omega_N$ and for any $z\in
H^1(\kij)$,
$\Pi^{-}z,\Pi_x^{+} z,\Pi_y^{+} z\in \mathcal{Q}^k(K_{ij})$
are defined as
\begin{subequations} \label{def:prj:GR}
	\begin{align}\label{global:projection:2d:u}
	&
	\begin{cases}
	\dual{\Pi^{-}z}{\vphu}_{\kij}
	= \dual{z}{\vphu}_{\kij},
	\quad &\forall \vphu\in \mathcal{Q}^{k-1}(\kij),
	\\
	\dual{(\Pi^{-}z)_{i,y}^{-}}{\vphu}_{J_j}
	=\dual{z_{i,y}^{-}}{\vphu}_{J_j},
	\quad &\forall \vphu\in \mathcal{P}^{k-1}(J_j),
	\\
	\dual{(\Pi^{-} z)_{x,j}^{-}}{\vphu}_{I_i}
	=\dual {z_{x,j}^{-}}{\vphu}_{I_i},
	\quad &\forall \vphu\in \mathcal{P}^{k-1}(I_i),
	\\
	(\Pi^{-} z)(x_{i}^{-},y_{j}^{-})=z(x_{i}^{-},y_{j}^{-}).
	\end{cases}
	\\
	&
	\label{global:projection:2d:q1}
	\begin{cases}
	\dual{\Pi_x^{+} z}{\vphu}_{\kij}
	=
	\dual{z}{\vphu}_{\kij},
	\quad &\forall \vphu\in \mathcal{P}^{k-1}(I_i)\otimes \mathcal{P}^k(J_j),
	\\
	\dual{(\Pi_x^{+} z)_{i,y}^{+}}{\vphu}_{J_j}
	=\dual{z_{i,y}^{+}}{\vphu}_{J_j},
	\quad &\forall \vphu\in \mathcal{P}^k(J_j).
	\end{cases}
	\\
	\label{global:projection:2d:q2}
	&
	\begin{cases}
	\dual{\Pi_y^{+}  z} {\vphu}_{\kij}
	=
	\dual{z}{\vphu}_{\kij}
	\quad &\forall \vphu\in \mathcal{P}^{k}(I_i)\otimes \mathcal{P}^{k-1}(J_j),
	\\
	\dual
	{(\Pi_y^{+} z)^{+}_{x,j}}{\vphu}_{J_j}
	=\dual{z^{+}_{x,j}}{\vphu}_{J_j},
	\quad &\forall \vphu\in \mathcal{P}^k(I_i).
	\end{cases}
	\end{align}
\end{subequations}

\begin{lemma} \label{2GR:inequality}\cite{Apel1999,Zhu:2dMC,Cheng2020}
	There exists a constant $C>0$, independent of the element size and $z$, such that
	\begin{subequations}\label{2GR:stb:app}
		\begin{align}
		\label{2GR:2:stb}
		\norm{\Pi^- z}_{\kij}\leq &\; C
		\big[
		\norm{z}_{\kij} + h_{j}\norm{z^-_{x,j}}_{I_i}
		+h_{i}\norm{z^-_{i,y}}_{J_j} + h_{i}h_{j}|z_{i,j}^-|
		\big],
		\\
		\norm{\Pi^+_x z}_{\kij}\leq &\; C
		\big[
		\norm{z}_{\kij} 
		+h_{i}\norm{z^+_{i,y}}_{J_j}
		\big],
		\\
		\norm{\Pi^+_y z}_{\kij}\leq &\; C
		\big[
		\norm{z}_{\kij} + h_{j}\norm{z^+_{x,j}}_{I_i}
		\big],
		\\
		\label{2GR:inf:stb}
		\norm{\Phi z}_{L^\infty(\kij)}\leq &\; C \norm{z}_{L^\infty(\kij)},
		\\
		\label{2GR:app}
		\norm{z-\Phi z}_{L^m(\kij)}\leq &\;
		C \Big[h_{i}^{k+1}\norm{\partial_x^{k+1}z}_{L^m(\kij)}
		+h_{j}^{k+1}\norm{\partial_y^{k+1}z}_{L^m(\kij)}
		\Big],
		\end{align}
	\end{subequations}
	where  $\Phi\in\{\Pi^{-},\Pi_x^{+},\Pi_y^{+}\}$, $m \in \{2,\infty\}$ and $z_{i,j}^-=z(x_i^-,y_j^-)$.
\end{lemma}

\begin{lemma}\label{lemma:GR}
	Let $\sigma\geq k+1.5$. Then, there exists a constant $C>0$  independent of $\varepsilon$ and $N$ such that
	\begin{subequations}\label{2dGR:property}
		\begin{align}
		\label{app:u:2}
		\norm{u-\Pi^- u}
		&\; \leq C\big[
		(\sfq+\sparep)(N^{-1}\max|\ppri|)^{k+1}+N^{-(k+1)}
		\big],
		\\
		\label{app:u:boundary}
		\sum_{j=1}^{N}
		\norm{(u-\Pi^- u)_{N,y}^{-}}^2_{J_j}
		&\;\leq 
		C\big[
		(\ssq+\parep)(N^{-1}\max|\ppri|)^{2(k+1)}+N^{-2(k+1)}
		\big],
		\\
		\label{app:pq:2}
		\invssq\norm{p-\Pi_x^+ p}
		&\;
		\leq C[(\sfq+\sparep)(N^{-1}\max|\psi^{\prime}|)^{k+1}
		+ N^{-(k+1)}],
		\\
		\label{app:pq:boundary}
		\sum_{j=1}^{N} \sq^{-1}
		\norm{(p-\Pi_x^+ p)_{N,y}^{-}}^2_{J_j}
		&\;\leq C (N^{-1}\max|\psi^{\prime}|)^{2(k+1)},
		\end{align}
	\end{subequations}
	where $\parep$ is given by \eqref{def:qstar}.	
	Similarly, we can obtain the same conclusions in another spatial direction. 
\end{lemma}

\begin{proof}
	The conclusions are more precise than Lemma 4.1 of \cite{Cheng2020}.
	The proof proceeds similarly to Lemma \ref{lemma:L2}.
	We mention several differences and use the same notations
	to prevent confusion.

	(1) Prove \eqref{app:u:2}.	
	As before, we have
	\begin{align}\label{eta:0}
	\norm{\eta_{S}} \leq C N^{-(k+1)}.
	\end{align}
	
	To bound $\norm{\eta_{W_1}}$, we express it as \eqref{w1}.
	Using \eqref{2GR:2:stb}
	and \eqref{2GR:app} with $m=2$, we obtain the two estimates
	\begin{align}\label{g11}
	\Lambda_{1}\leq &\; C\sum_{\kij \in \xl}
	\Big[ \norm{W_1}_{\kij}^2 + h_{j}\norm{(W_1)^-_{x,j}}_{I_i}^2
	+h_{i}\norm{(W_1)^-_{i,y}}_{J_j}^2 + h_{i}h_{j}|(W_1)^-_{i,j}|^2 \Big]
	\nonumber\\
	\leq &\; C\sum_{\kij \in \xl}\norm{e^{-\frac{\beta x}{\ssq}}}_{\kij}^2,
	\\
	\label{g12}
	\Lambda_{1}\leq &\; C\sum_{\kij \in \xl}
	\Big[
	h_i^{2(k+1)} \norm{\dx^{k+1}W_1}_{\kij}^2+
	h_j^{2(k+1)} \norm{\dy^{k+1}W_1}_{\kij}^2
	\Big]
	\nonumber\\
	\leq &\; C \sum_{\kij \in \xl}
	\Big[
	\Big(\frac{h_i}{\ssq}\Big)^{2(k+1)}+h_j^{2(k+1)}
	\Big]
	\norm{e^{-\frac{\beta x}{\ssq}}}_{\kij}^2,
	\end{align}
	respectively, where we used \eqref{reg:u:2d} and the monotonic decreasing property of the function $e^{-\beta x/\ssq}$. 
	Consequently, we obtain the same estimate
	for $\Lambda_{1}$ as before.
	For $\Lambda_2$, we use the stability \eqref{2GR:2:stb}
	and \eqref{reg:u:2d} to obtain
	\begin{align} \label{A2}
	\Lambda_2= &\;
	C\sum_{\kij \in \xm\cup\xr}
	\Big[ \norm{W_1}_{\kij}^2 + h_{j}\norm{(W_1)^-_{x,j}}_{I_i}^2
	+h_{i}\norm{(W_1)^-_{i,y}}_{J_j}^2 + h_{i}h_{j}|(W_1)^-_{i,j}|^2 \Big]
	\nonumber\\
	\leq &\;
	C\sum_{\kij \in \xm\cup\xr}\norm{e^{-\frac{\beta x}{\ssq}}}_{\kij}^2
	=
	C\int_{0}^{1}dy\int_{\tau}^{1} e^{-\frac{2\beta x}{\ssq} } dx
	\leq  C\ssq e^{-\frac{2\beta \tau}{\ssq}}
	\leq C\ssq N^{-2\sigma}.
	\end{align}	
	As a result, we have
	\begin{align}\label{eta:1}
	\norm{\eta_{W_1}}
	\leq C\sfq[(N^{-1}\max|\psi^{\prime}|)^{k+1}+N^{-\sigma}].
	\end{align}
	
	The term $\norm{\eta_{W_3}}$ must be treated carefully.
	We decompose it as
	\[
	\norm{\eta_{W_3}}^2
	=\sum_{\kij \in \xl} \norm{\eta_{W_3}}_{\kij}^2
	+\sum_{\kij \in \xm} \norm{\eta_{W_3}}_{\kij}^2
	+\sum_{\kij \in \xr}\norm{\eta_{W_3}}_{\kij}^2
	\equiv \Gamma_1 + \Gamma_2 + \Gamma_3.
	\]
	
	Using $L^\infty$-stability and \eqref{reg:u:2d}, we have
	\begin{align}\label{B1}
	\Gamma_1 \leq &\; 
	C\sum_{\kij \in \xl} 
	h_ih_j \norm{W_3}^2_{L^\infty(\kij)} 
	\leq  C\norm{W_3}^2_{L^\infty(\xl)} 
	\sum_{i=1}^{\ofh} h_i
	\sum_{j=1}^{N} h_j
	\nonumber\\
	\leq &\;
	C  \norm{W_3}^2_{L^\infty(\xl)} 
	\leq C e^{-\frac{2\beta (1-\tau)}{\ssq}}
	\leq C\ssq e^{-\frac{2\beta\tau}{\ssq}}
	\leq C\ssq N^{-2\sigma},
	\end{align}
	because $2(1-\tau)\geq 1+2\tau$ for $0\leq \tau\leq 1/4$,
	and $e^{-x}<x^{-1}$ for $x\geq 1$.

	Using the $L^2$-stability, the uniformity of the mesh in $\xm$, and
	Proposition \ref{thm:reg:2d}, we have
	\begin{align} \label{B2}
	\Gamma_2 \leq &\;
	C\sum_{\kij \in \xm}
	\Big[ \norm{W_3}^2_{\kij} + h_{j}\norm{(W_3)^-_{x,j}}^2_{I_i}
	+h_{i}\norm{(W_3)^-_{i,y}}^2_{J_j} + h_{i}h_{j}|(W_3)_{i,j}^-|^2 \Big]
	\nonumber\\
	\leq &\;  C\sum_{\kij \in \xm}\norm{e^{-\frac{\beta x}{\ssq}}}_{\kij}^2
	+ C\sum_{j=1}^{N} h_j h_{3N/4} e^{-\frac{2\beta(1-x_{3N/4})}{\ssq}}
	\nonumber\\
	\leq &\;
	C\ssq e^{-\frac{2\beta \tau}{\ssq}}
	+N^{-1} e^{-\frac{2\beta\tau}{\ssq}}
	\leq  C(\ssq+N^{-1}) N^{-2\sigma}.
	\end{align}

	To bound $B_3$, we decompose it into two parts:
	\begin{align} \label{B3:decompose}
	\Gamma_3	= \sum_{i=\tfh+2}^{N-1} \sum_{j=1}^{N} \norm{\eta_{W_3}}^2_{\kij}
	+\sum_{i\in\{\tfh+1,N\}} \sum_{j=1}^{N}
	\norm{\eta_{W_3}}^2_{\kij}
	\equiv \Gamma_3^{(1)} + \Gamma_3^{(2)}.
	\end{align}
	From \eqref{2GR:2:stb}, \eqref{2GR:inf:stb}, and Lemma \ref{lem:h:min:max}, 
	we have
	\begin{align} \label{B3:stb:1}
	\Gamma_3^{(1)} \leq &\;C\sum_{i=\tfh+2}^{N-1} \sum_{j=1}^{N}
	\Big[ \norm{W_3}^2_{\kij} + h_{j}\norm{(W_3)^-_{x,j}}^2_{I_i}
	+h_{i}\norm{(W_3)^-_{i,y}}^2_{J_j} + h_{i}h_{j}|(W_3)^-_{i,j}|^2\Big]
	\nonumber\\
	\leq &\; 
	C\sum_{\kij \in \xr}
	\norm{e^{-\frac{2\beta (1-x)}{\ssq} }}^2_{\kij}.
	\end{align}
	Using \eqref{2GR:app} with $m=2$ yields
	\begin{align}   \label{B3:app:1}
	\Gamma_{3}^{(1)}\leq &\; C\sum_{\kij \in \xr}
	\Big[
	h_i^{2(k+1)} \norm{\dx^{k+1}W_3}_{\kij}^2+
	h_j^{2(k+1)} \norm{\dy^{k+1}W_3}_{\kij}^2
	\Big]
	\nonumber\\
	\leq &\; C \sum_{\kij \in \xr}
	\Big[
	\Big(\frac{h_i}{\ssq}\Big)^{2(k+1)}+h_j^{2(k+1)}
	\Big]
	\norm{e^{-\frac{\beta (1-x)}{\ssq}}}_{\kij}^2.
	\end{align}	
	Then, combining \eqref{B3:stb:1} and \eqref{B3:app:1}, we obtain 
	\begin{align}
	\Gamma_3^{(1)} \label{B3:1}
	\leq &\; C\sum_{i=\frac{3N}{4}+1}^{N}
	\min\Big\{\Big(\frac{h_i}{\ssq}\Big)^{2(k+1)},1\Big\}
	\ssq e^{-\frac{2\beta (1-x_{i})}{\ssq}}
	\min\Big\{\frac{h_i}{\ssq},1\Big\}
	\nonumber\\
	\leq &\;
	C\ssq(N^{-1}\max|\ppri|)^{2(k+1)},
	\end{align}
	as before.
	
	In a similar manner to \eqref{w1:0}, we have
	\begin{align*}
	\norm{\eta_{W_3}}_{L^\infty(\kij)}^2 
	\leq&\; C \Theta_i^{2(k+1)}
	\leq C (N^{-1}\max|\ppri|)^{2(k+1)}.
	\end{align*}
	Because $h_{i} \leq  C \parep$ for $i=3N/4,...,N$, we have
	\begin{align} \label{B3:2}
	\Gamma_3^{(2)}  &\;
	\leq
	\sum_{i\in\{\tfh+1,N\}} \sum_{j=1}^{N}
	h_ih_j \norm{\eta_{W_3}}^2_{L^\infty(\kij)}
	\leq 
	C \parep (N^{-1}\max|\ppri|)^{2(k+1)}.
	\end{align}
	
	Combining \eqref{B3:1} with \eqref{B3:2} leads to
	\begin{align} \label{B3}
	\Gamma_3 \leq C (\ssq+\parep)(N^{-1}\max|\ppri|)^{2(k+1)}.
	\end{align}
	
	Collecting up \eqref{B1}, \eqref{B2}, and \eqref{B3} yields
	\begin{align}\label{eta:2}
	\norm{\eta_{W_3}}
	\leq C \Big[(\sfq+\sparep) (N^{-1}\max|\ppri|)^{k+1}+ (\sfq+ N^{-1/2}) N^{-(k+1)}\Big].
	\end{align}
	Similarly, we can prove the remainder of \eqref{eta:L2norm} and arrive at \eqref{app:u:2}.

	(2) 
	Note that $(\eta_u)_{N,y}^{-}=u_{N,y}-\pi_y^{-}(u_{N,y})$,
	where $\pi_y^{-}$ is a one-dimensional Gauss--Radau projection
	regarding $y$ and 
	satisfies analogous stability and approximation conditions to that in Lemma \ref{2GR:inequality}. From the solution decomposition,
	we express  $u_{N,y} = S_{N,y} +E_{N,y}+F_{N,y}$,
	where $S_{N,y},E_{N,y}$ and $F_{N,y}$ are functions
	of one variable $y$ and satisfy 
	$|S_{N,y}^{(j)}|\leq C$,
	$| E_{N,y}^{(j)}|\leq C\sq^{-j/2} e^{-\beta y/\ssq}$, and
	$| F_{N,y}^{(j)}|\leq C\sq^{-j/2} e^{-\beta (1-y)/\ssq}$.
	Following the similar line to that used to prove \eqref{2GR:2:stb},
	we obtain
	\begin{align*}
	\sum_{j=1}^{N}\norm{(\eta_{S})_{N,y}^{-}}_{J_j}^2
	&\; \leq N^{-2(k+1)},
	\\
	\sum_{j=1}^{N}\norm{(\eta_{E})_{N,y}^{-}}_{J_j}^2
	&\; \leq C\ssq[(N^{-1}\max|\psi^{\prime}|)^{2(k+1)}+N^{-2\sigma}],
	\\
	\sum_{j=1}^{N}\norm{(\eta_{F})_{N,y}^{-}}_{J_j}^2
	&\; \leq  
	C \Big[(\ssq+\parep) (N^{-1}\max|\ppri|)^{2(k+1)}+ (\ssq+ N^{-1}) N^{-2(k+1)}\Big].
	\end{align*}
	Using the triangle inequality leads to \eqref{app:u:boundary}.
	
	The proofs of \eqref{app:pq:2} and \eqref{app:pq:boundary} are similar and therefore omitted.
	
\end{proof}

\begin{theorem} \label{thm:energy}
	Suppose that $\lambda_{i,y}=\lambda_{x,j}=0$ for $i,j=0,1,\cdots,N-1 $, $\lambda_{N,y}=\lambda_{x,N}=\ssq$.
	Let $\bm w=(u,p,q)$ be the solution to problem \eqref{spp:R-D}, which
	satisfies Proposition \ref{thm:reg:2d}; furthermore,
	let $\wN=(\uN,\pN,\qN)\in \spc^3$ be the numerical solution of
	the LDG scheme \eqref{LDG:scheme:2d} on layer-adapted meshes \eqref{layer-adapted} with $\sigma\geq k+1.5$.
	Then, there exists a constant $C>0$ independent of $\sq$ and $N$ such that
	\begin{align}\label{error:energy:general}
	\enorm{\bm w-\wN} \leq
	\begin{dcases}
	C\Big[\sfq(N^{-1}\ln N)^{k+1}+N^{-(k+1)}\Big], & \textrm{for S-mesh},\\
	CN^{-(k+1)}, & \textrm{for BS-,B-type mesh}.
	\end{dcases}
	\end{align}
\end{theorem}

\begin{proof}
	We follow the proof of Theorem \ref{thm:balanced}.
	Instead of the projection $(\Pi_b,\Pi,\Pi)$ for $u,p$ and $q$,
	we use $(\Pi^-,\Pi_x^+,\Pi_y^+)$
	in \eqref{error:decomposition}.
	
	Using the Cauchy--Schwarz inequality, \eqref{app:u:2}, and \eqref{app:pq:2}, we obtain
	\begin{align} \label{appr:t1}
	\Ton{\bm \eta}{\vph}
	\leq &\;
	C\big(\varepsilon^{-1/2}\norm{\eta_p}+\varepsilon^{-1/2}\norm{\eta_q}
	+\norm{b^{1/2}}_{L^\infty(\Omega_N)}\norm{\eta_u}\big)\enorm{\vph}
	\nonumber\\
	\leq &\;
	C\Big[
	(\sfq+\sparep)(N^{-1}\max|\ppri|)^{k+1}+N^{-(k+1)}
	\Big]\enorm{\vph}.
	\end{align}
	
	From \eqref{global:projection:2d:q1},
	\eqref{global:projection:2d:q2},
	\eqref{app:pq:boundary}, $\lambda_{N,y}=\lambda_{x,N}=\ssq$, and the Cauchy--Schwarz inequality, we obtain
	\begin{align}\label{appr:t3} 
	\Tth{\bm \eta}{\vphu}
	=&\;
	-\sum_{j=1}^{N}
	\dual{(\eta_p)^{-}_{N,y}}{\vphu^{-}_{N,y}}_{J_j}
	-\sum_{i=1}^{N}
	\dual{(\eta_q)^{-}_{x,N}}{\vphu^{-}_{x,N}}_{I_i}
	\nonumber\\
	\leq &\; 
	\Big[
	\sum_{j=1}^{N}\frac{1}{\lambda_{N,y}}\norm{(\eta_p)_{N,y}^{-}}^2_{J_j}
	+\sum_{i=1}^{N}\frac{1}{\lambda_{x,N}}\norm{(\eta_q)_{x,N}^{-}}^2_{I_i}
	\Big]^{1/2} \enorm{\vph}
	\nonumber\\
	\leq &\;C\sfq(N^{-1}\max|\psi^{\prime}|)^{k+1} \enorm{\vph}.
	\end{align}  
	Similarly, we have
	\begin{align} \label{appr:t4}
	\Tfr{\eta_u}{\vphu} = 
	&\;
	\sum_{j=1}^{N}
	\dual{\lambda_{N,y}(\eta_u)^{-}_{N,y}}{\vphu^{-}_{N,y}}_{J_j}
	+
	\sum_{i=1}^{N}
	\dual{\lambda_{x,N}(\eta_u)^{-}_{x,N}}{\vphu^{-}_{x,N}}_{I_i}
	\nonumber\\
	\leq &\;
	\Big[
	\sum_{j=1}^{N}
	\lambda_{N,y} \norm{(\eta_u)_{N,y}^{-}}^2_{J_j}
	+
	\sum_{i=1}^{N}
	\lambda_{x,N} \norm{(\eta_u)_{x,N}^{-}}^2_{I_i}
	\Big]^{1/2} \enorm{\vph}
	\nonumber\\
	\leq &\;
	C\sfq\Big[
	(\sfq+\sparep)(N^{-1}\max|\ppri|)^{k+1}+N^{-(k+1)}
	\Big] \enorm{\vph},
	\end{align}
	where \eqref{app:u:boundary} was used.

	To bound $\Ttw{\eta_u}{\vph}$,
	we follow \cite{Cheng2020} and investigate the $\sq^{1/4}$-factor in the upper-bound.
	On each element $K_{ij}$, we define the bilinear forms as
	\begin{align*}
	\mathcal{A}^1_{ij}(\eta_u,v)
	=&\;\dual{\eta_u}{v_x}_{K_{ij}}
	-\dual{(\eta_u)^{-}_{i,y}}{v^{-}_{i,y}}_{J_{j}}
	+\dual{(\eta_u)^{-}_{i-1,y}}{v^{+}_{i-1,y}}_{J_{j}},
	\\
	\mathcal{A}^2_{ij}(\eta_u,v)
	=&\;\dual{\eta_u}{v_y}_{K_{ij}}
	-\dual{(\eta_u)^{-}_{x,j}}{v^{-}_{x,j}}_{I_{i}}
	+\dual{(\eta_u)^{-}_{x,j-1}}{v^{+}_{x,j-1}}_{I_{i}}.
	\end{align*}	
	We have \cite{Cheng2020}
	\begin{subequations}\label{2d:sup}
		\begin{align}
		\label{2d:sup:1}
		|\mathcal{A}^1_{ij}(\eta_u,v)|\leq&
		C\sqrt{\frac{h_{j}}{h_{i}}}
		\Big[h_{i}^{k+2}
		\norm{\partial_x^{k+2}u}_{L^\infty(K_{ij})}
		+h_{j}^{k+2}\norm{\partial_y^{k+2}u}_{L^\infty(K_{ij})}
		\Big]\norm{v}_{K_{ij}},
		\\
		\label{2d:sup:2}
		|\mathcal{A}^1_{ij}(\eta_u,v)|\leq&
		C\sqrt{\frac{h_{j}}{h_{i}}}\norm{u}_{L^\infty(K_{ij})}\norm{v}_{K_{ij}}
		\end{align}		
	\end{subequations}
	for any $v\in \mathcal{Q}^k(K_{ij})$. Because $u=0$ on $\partial\Omega$, we obtain
	\begin{align*}
	\Ttw{\eta_u}{\vph}
	&\;=\sum_{K_{ij}\in \Omega_N}\mathcal{A}^1_{ij}(\eta_u,\vphp)
	+\sum_{K_{ij}\in \Omega_N}\mathcal{A}^2_{ij}(\eta_u,\vphq).
	\end{align*}
	
	By \eqref{2d:sup:1}, we have
	\begin{align*}
	&\sum_{K_{ij}\in \Omega_N}\mathcal{A}^1_{ij}(\eta_S,\vphp)
	\nonumber\\
	\leq&\; C\sum_{K_{ij}\in \Omega_N}\sqrt{\frac{h_{j}}{h_{i}}}
	\Big[h_{i}^{k+2}
	\norm{\partial_x^{k+2}S}_{L^\infty(K_{ij})}
	+h_{j}^{k+2}\norm{\partial_y^{k+2}S}_{L^\infty(K_{ij})}
	\Big]\norm{\vphp}_{K_{ij}}
	\\
	\leq&\; CN^{-(k+1)}\norm{\vphp}_{\Omega_{22}}
	+C\Bigg[\sum_{K_{ij}\in \Omega_N\setminus\Omega_{22}}
	(\ssq \max|\psi^{\prime}|)^{-1} N^{-2(k+2)}
	\Bigg]^{1/2}\norm{\vphp}_{\Omega\setminus\Omega_{22}}
	\\
	\leq&\; C\sfq N^{-(k+1)}\enorm{\vph},
	\end{align*}
	because $h_{i}\geq C\ssq N^{-1}\max|\psi^{\prime}|\geq C\ssq h_{j}\max|\psi^{\prime}|$
	and $\varepsilon^{-1/2}\norm{\vphp}\leq \enorm{\vph}$.
	Using \eqref{2d:sup:1} and $\sigma\geq k+2$, we obtain
	\begin{align*}
	\sum_{K_{ij}\in \xm\cup \xr}
	\mathcal{A}^1_{ij}(\eta_{W_1},\vphp)
	\leq&\; C\sum_{K_{ij}\in \xm\cup \xr}
	\sqrt{\frac{h_{j}}{h_{i}}}\norm{W_1}_{L^\infty(K_{ij})}\norm{\vphp}_{K_{ij}}
	\nonumber\\
	\leq&\; C\sum_{K_{ij}\in \xm\cup \xr}
	\sq^{-1/4} e^{-\frac{\beta x_i}{\ssq}}\norm{\vphp}_{K_{ij}}
	\nonumber\\
	\leq&\; C\sfq N^{-(k+1)}\enorm{\vph}.
	\end{align*}
	
	Using \eqref{2d:sup}, $\sigma\geq k+2$, and Lemma \ref{lemma:1} yields
	\begin{align*}
	&\sum_{K_{ij}\in \xl}\mathcal{A}^1_{ij}(\eta_{W_1},\vphp)
	\\
	\leq&\; C\sum_{K_{ij}\in \xl}
	\sqrt{\frac{h_{j}}{h_{i}}}
	\min\Big\{
	h_{i}^{k+2}\norm{\partial_x^{k+2}W_1}_{L^\infty(K_{ij})}
	+h_{j}^{k+2}\norm{\partial_y^{k+2}W_1}_{L^\infty(K_{ij})},
	\norm{W_1}_{L^\infty(K_{ij})}
	\Big\}\norm{\vphp}_{K_{ij}}
	\nonumber\\
	\leq &\;
	C\sum_{K_{ij}\in \xl}
	(\ssq \max|\psi^{\prime}|)^{-1/2}\Theta_i^{k+2}\norm{\vphp}_{K_{ij}}
	\nonumber\\
	\leq &\;
	C\sfq(\max|\psi^{\prime}|)^{-1/2}\Bigg(\sum_{j=1}^N
	\sum_{i=1}^{\ofh}\Theta_i\Bigg)^{1/2}
	\Big(\max_{1\leq i \leq \ofh}\Theta_i^{k+3/2}\Big)
	\Big(\sq^{-1/2}\norm{\vphp}\Big)
	\nonumber\\
	\leq&\; C\sfq(N^{-1}\max|\ppri|)^{k+1}\enorm{\vph}.
	\end{align*}
	Analogously, we can bound
	$\sum_{K_{ij}\in \Omega_{N}}\mathcal{A}^1_{ij}(\eta_\varphi,\vphp)$
	for $\varphi=W_2,W_3,W_4,Z_1,Z_2,Z_3,Z_4$.
	Consequently,
	\begin{align} \label{appr:t2}
	\Ttw{\eta_{u}}{\vph} \leq C\sfq(N^{-1}\max|\psi^{\prime}|)^{k+1}\enorm{\vph}.
	\end{align}
	Using \eqref{appr:t1}, \eqref{appr:t3}, \eqref{appr:t4}, and \eqref{appr:t2}, we obtain
	\begin{align*} 
	\enorm{\bm \xi}^2 = \Bln{\bm\xi}{\bm\xi} = \Bln{\bm\eta}{\bm\xi} 
	\leq 	C\Big[
	(\sfq+\sparep)(N^{-1}\max|\ppri|)^{k+1}+N^{-(k+1)}
	\Big] \enorm{\bm \xi},
	\end{align*}
	which leads to
	\begin{align}
	\enorm{\bm \xi} \leq \;
	C\Big[
	(\sfq+\sparep)(N^{-1}\max|\psi^{\prime}|)^{k+1} + N^{-(k+1)}
	\Big].
	\end{align}
	The final assertion follows by repeating similar arguments as before.
	This completes the proof.
\end{proof}

\section{Numerical experiments}
\label{sec:experiments}

In this section, we present some numerical experiments.
All calculations were conducted in MATLAB R2015B.
The system of linear equations resulting from the discrete problems
were solved by the lower--upper (LU)-decomposition algorithm. All integrals were evaluated
using the 5-point Gauss--Legendre quadrature rule.

The LDG method \eqref{LDG:scheme:2d} was applied
to the layer-adapted meshes presented in Table \ref{table:functions},
where $\sigma=k+1$, $k=0,1,2,3$.
We let $\bm e^N$ be the error in either $\enorm{\bm e}$ or $\ba{\bm e}$ for an $N$-element.
In the former case, we took the flux parameter $\lambda_i = \ssq$
for $i=0,N$ and $\lambda_i = 0$ for $i=1,2,\dots,N-1$.
In the last case, we took the flux parameter $\lambda_i = \ssq$
for $i=0,1,\dots,N$.
The corresponding convergence rates were computed by the following formulae:
\begin{align*}
r_S =
\frac{ \log\bm e^N-\log\bm e^{2N} }{\log p },
\quad
r_2 =
\frac{ \log\bm e^N-\log\bm e^{2N} }{\log 2 }.
\end{align*}
Here, $p = 2\ln N/\ln(2N)$ was used 
to compute the numerical convergence order
with respect to the power of $\ln N/N$.

\textbf{Example 1.}
Consider a linear reaction--diffusion problem 
\begin{subequations}
	\begin{align}
	-\varepsilon \Delta u +  2u
	&= f(x,y),   &\textrm{in}&\;  \Omega =(0,1)^2,
	\\
	u &= 0,   &\textrm{on}&\; \partial \Omega,
	\end{align}
\end{subequations}
where $f(x,y)$ is suitably taken such that the exact solution is 
$u(x,y) = g(x)g(y)$ with
\begin{align}\label{solution:cos}
g(v)=\frac{e^{-v/\ssq}-e^{-(1-v)/\ssq}}{1-e^{-1/\ssq}}-\cos(\pi v).
\end{align}

We set $\varepsilon=10^{-8}$, small enough
to bring out the singularly perturbed nature of \eqref{solution:cos}.
In Table \ref{table:stb:balanced},
we list the balanced-norm errors and their convergence rates.
We observed convergence of order $k+1/2$,
which is a half-order superior to the estimate from \eqref{error:balanced:general}.
In Table \ref{table:stb:energy},
we present the energy norm errors and their convergence rates,
which agree with our estimate from \eqref{error:energy:general}.

We show the relevance of these errors to the small parameter $\varepsilon$.
We let $N=256$, $k=1$, and varied the values of $\varepsilon$.
From Table \ref{table:stb:eps}, we see that the errors in the balanced norm are almost unchanged, whereas the errors in the energy norm change slightly.
For a visual understanding,
we plotted the energy errors via $\sq$ on log--log coordinates.
In Figure \ref{fig:ep:err:linear},
we observe the subtle influence of the $\sq^{0.25}$-factor on the energy errors; the results agree with our predictions.

\textbf{Example 2.}
Consider the nonlinear--reaction-diffusion problem 
\begin{subequations}
	\begin{align}
	-\varepsilon \Delta u +  [2+xy(1-x)(1-y)] u
	&= f(x,y),   &\textrm{in}&\;  \Omega =(0,1)^2,
	\\
	u &= 0,   &\textrm{on}&\; \partial \Omega,
	\end{align}
\end{subequations}
where $f(x,y)$ is suitably taken such that the exact solution is 
$u(x,y) = h(x)h(y)$ and
\begin{align}\label{solution:nonlinear}
h(v)=1+(v-1)e^{-v/\ssq}-ve^{-(1-v)/\ssq}.
\end{align}

We let $\varepsilon=10^{-8}$. In Table \ref{table:stb:balanced:nonlinear} and Table \ref{table:stb:energy:nonlinear}, we list the error and convergence rates for the balanced and energy norms, respectively. We still observed convergence of orders $k+1/2$ and $k+1$ for the balanced-norm and energy norm errors.

Moreover, we tested the dependence of these two types of errors on $\varepsilon$. We clearly observed in Table \ref{table:eps:nonlinear} that the errors in the balanced norm were almost constant, whereas the errors in the energy norm changed slightly. In Figure \ref{fig:ep:err:nonlinear}, we confirmed the influence of the factor $\varepsilon^{0.25}$ on the energy norm errors obtained for the three layer-adapted meshes.
Note that for this example, the 
$\varepsilon^{0.25}$-factor is clearly observed on both the S-type and B-type meshes.
This may be due to the fact that the regular part of the exact solution
belongs to $\spc$, as described in \cite{Zhu:1dCMS}.

\begin{table}[ht]
	\normalsize
	\centering
	\caption{Balanced error and convergence rates for Example 1.}
	\label{table:stb:balanced}
	\begin{tabular}{cccccccc}
		\toprule
		\multirow{2}*{$k$}&\multirow{2}*{$N$} & \multicolumn{2}{c}{S-mesh} & \multicolumn{2}{c}{BS-mesh} & \multicolumn{2}{c}{B-mesh} \\
		\cmidrule(r){3-4} \cmidrule(r){5-6} \cmidrule(r){7-8}
		&& Balanced error & $r_s$ & Balanced error & $r_2$ & Balanced error & $r_2$
		\\
		\midrule
		0
		& 8 & 1.37e+00 & - & 1.36e+00 & - & 1.55e+00 & - \\
		& 16 & 1.09e+00 & 0.57 & 1.04e+00 & 0.39 & 1.10e+00 & 0.49 \\
		& 32 & 8.36e-01 & 0.57 & 7.47e-01 & 0.47 & 7.67e-01 & 0.52 \\
		& 64 & 6.31e-01 & 0.55 & 5.30e-01 & 0.50 & 5.36e-01 & 0.52 \\
		& 128 & 4.73e-01 & 0.54 & 3.74e-01 & 0.50 & 3.76e-01 & 0.51 \\
		& 256 & 3.52e-01 & 0.53 & 2.64e-01 & 0.50 & 2.65e-01 & 0.51 \\
		1
		& 8 & 3.67e-01 & - & 2.48e-01 & - & 3.86e-01 & - \\
		& 16 & 2.22e-01 & 1.25 & 9.83e-02 & 1.33 & 1.22e-01 & 1.66 \\
		& 32 & 1.19e-01 & 1.32 & 3.75e-02 & 1.39 & 4.17e-02 & 1.55 \\
		& 64 & 5.83e-02 & 1.40 & 1.39e-02 & 1.43 & 1.46e-02 & 1.51 \\
		& 128 & 2.68e-02 & 1.44 & 5.04e-03 & 1.46 & 5.16e-03 & 1.50 \\
		& 256 & 1.18e-02 & 1.46 & 1.81e-03 & 1.48 & 1.83e-03 & 1.49 \\
		2
		& 8 & 1.61e-01 & - & 7.24e-02 & - & 1.45e-01 & - \\
		& 16 & 7.40e-02 & 1.92 & 1.58e-02 & 2.20 & 2.26e-02 & 2.69 \\
		& 32 & 2.68e-02 & 2.16 & 3.11e-03 & 2.35 & 3.71e-03 & 2.61 \\
		& 64 & 8.19e-03 & 2.32 & 5.83e-04 & 2.42 & 6.34e-04 & 2.55 \\
		& 128 & 2.23e-03 & 2.42 & 1.06e-04 & 2.45 & 1.11e-04 & 2.52 \\
		& 256 & 5.62e-04 & 2.46 & 1.91e-05 & 2.48 & 1.95e-05 & 2.51 \\
		3
		& 8 & 7.16e-02 & - & 2.16e-02 & - & 5.87e-02 & - \\
		& 16 & 2.52e-02 & 2.57 & 2.52e-03 & 3.10 & 4.22e-03 & 3.80 \\
		& 32 & 6.29e-03 & 2.96 & 2.55e-04 & 3.30 & 3.29e-04 & 3.68 \\
		& 64 & 1.21e-03 & 3.23 & 2.42e-05 & 3.40 & 2.74e-05 & 3.59 \\
		& 128 & 1.96e-04 & 3.38 & 2.22e-06 & 3.45 & 2.35e-06 & 3.54 \\
		& 256 & 2.85e-05 & 3.45 & 2.01e-07 & 3.47 & 2.06e-07 & 3.51 \\
		\bottomrule
	\end{tabular}
\end{table}

\begin{table}[ht]
	\normalsize
	\centering
	\caption{Energy error and convergence rates for Example 1.}
	\label{table:stb:energy}
	\begin{tabular}{cccccccc}
		\toprule
		\multirow{2}*{$k$}&\multirow{2}*{$N$} & \multicolumn{2}{c}{S-mesh} & \multicolumn{2}{c}{BS-mesh} & \multicolumn{2}{c}{B-mesh} \\
		\cmidrule(r){3-4} \cmidrule(r){5-6} \cmidrule(r){7-8}
		&& Energy error & $r_S$ & Energy error & $r_2$ & Energy error & $r_2$
		\\
		\midrule
		0
		& 8 & 2.22e-01 & - & 2.22e-01 & - & 2.21e-01 & - \\
		& 16 & 1.13e-01 & 1.67 & 1.13e-01 & 0.96 & 1.13e-01 & 0.98 \\
		& 32 & 5.67e-02 & 1.46 & 5.66e-02 & 0.94 & 5.66e-02 & 0.99 \\
		& 64 & 2.84e-02 & 1.35 & 2.83e-02 & 0.99 & 2.83e-02 & 1.00 \\
		& 128 & 1.43e-02 & 1.28 & 1.42e-02 & 1.00 & 1.42e-02 & 1.00 \\
		& 256 & 7.15e-03 & 1.23 & 7.08e-03 & 1.00 & 7.08e-03 & 1.00 \\
		1
		& 8 & 2.30e-02 & - & 2.29e-02 & - & 2.30e-02 & - \\
		& 16 & 6.06e-03 & 3.29 & 5.77e-03 & 1.99 & 5.81e-03 & 1.98 \\
		& 32 & 1.73e-03 & 2.67 & 1.45e-03 & 1.99 & 1.46e-03 & 1.99 \\
		& 64 & 5.40e-04 & 2.27 & 3.64e-04 & 2.00 & 3.66e-04 & 2.00 \\
		& 128 & 1.77e-04 & 2.07 & 9.12e-05 & 2.00 & 9.17e-05 & 2.00 \\
		& 256 & 5.81e-05 & 1.99 & 2.29e-05 & 1.99 & 2.30e-05 & 1.99 \\
		2
		& 8 & 2.20e-03 & - & 1.66e-03 & - & 2.23e-03 & - \\
		& 16 & 7.06e-04 & 2.80 & 2.26e-04 & 2.87 & 2.89e-04 & 2.95 \\
		& 32 & 2.24e-04 & 2.44 & 3.05e-05 & 2.89 & 3.72e-05 & 2.96 \\
		& 64 & 6.01e-05 & 2.58 & 4.04e-06 & 2.92 & 4.77e-06 & 2.96 \\
		& 128 & 1.39e-05 & 2.72 & 5.30e-07 & 2.93 & 6.08e-07 & 2.97 \\
		& 256 & 2.85e-06 & 2.83 & 6.90e-08 & 2.94 & 7.74e-08 & 2.97 \\
		3
		& 8 & 7.11e-04 & - & 2.19e-04 & - & 6.89e-04 & - \\
		& 16 & 2.31e-04 & 2.77 & 2.03e-05 & 3.43 & 4.37e-05 & 3.98 \\
		& 32 & 5.27e-05 & 3.14 & 1.60e-06 & 3.66 & 2.75e-06 & 3.99 \\
		& 64 & 9.15e-06 & 3.43 & 1.16e-07 & 3.79 & 1.73e-07 & 3.99 \\
		& 128 & 1.29e-06 & 3.63 & 8.05e-09 & 3.85 & 1.08e-08 & 3.99 \\
		& 256 & 1.56e-07 & 3.77 & 5.43e-10 & 3.89 & 6.80e-10 & 3.99 \\
		\bottomrule
	\end{tabular}
\end{table}

\begin{table}[htp]
	\normalsize
	\centering
	\caption{ Energy norm and balanced-norm errors
		for different $\varepsilon$.}
	\label{table:stb:eps}
	\begin{tabular}{ccccccc}
		\toprule
		\multirow{2}*{$\varepsilon$}& \multicolumn{3}{c}{Energy  error} &\multicolumn{3}{c}{Balanced error}\\
		\cmidrule(lr){2-4} \cmidrule(lr){5-7}
		& S-mesh & BS-mesh & B-mesh & S-mesh & BS-mesh & B-mesh\\
		\midrule
		$10^{-6}$  & 1.71e-04 & 2.94e-05 & 3.02e-05
		& 1.18e-02 & 1.85e-03 & 1.87e-03\\
		$10^{-7}$  & 9.79e-05 & 2.42e-05 & 2.45e-05
		& 1.18e-02 & 1.83e-03 & 1.84e-03\\
		$10^{-8}$  & 5.81e-05 & 2.29e-05 & 2.30e-05 
		& 1.18e-02 & 1.81e-03 & 1.83e-03\\
		$10^{-9}$  & 3.76e-05 & 2.26e-05 & 2.26e-05 
		& 1.18e-02 & 1.81e-03 & 1.83e-03\\
		$10^{-10}$  & 2.81e-05 & 2.25e-05 & 2.25e-05 
		& 1.18e-02 & 1.81e-03 & 1.83e-03\\
		$10^{-11}$  & 2.44e-05 & 2.25e-05 & 2.25e-05 
		& 1.18e-02 & 1.81e-03 & 1.83e-03\\
		$10^{-12}$  & 2.31e-05 & 2.25e-05 & 2.25e-05 
		& 1.18e-02 & 1.81e-03 & 1.83e-03\\
		$10^{-13}$  & 2.27e-05 & 2.25e-05 & 2.25e-05 
		& 1.18e-02 & 1.81e-03 & 1.83e-03\\
		$10^{-14}$  & 2.25e-05 & 2.25e-05 & 2.25e-05 
		& 1.18e-02 & 1.81e-03 & 1.83e-03\\
		$10^{-15}$  & 2.25e-05 & 2.25e-05 & 2.25e-05 
		& 1.18e-02 & 1.81e-03 & 1.83e-03\\
		$10^{-16}$  & 2.25e-05 & 2.25e-05 & 2.25e-05
		& 1.18e-02 & 1.81e-03 & 1.83e-03\\
		\bottomrule
	\end{tabular}
\end{table}

\begin{figure}
	\centering
	\includegraphics[width=0.7\linewidth]{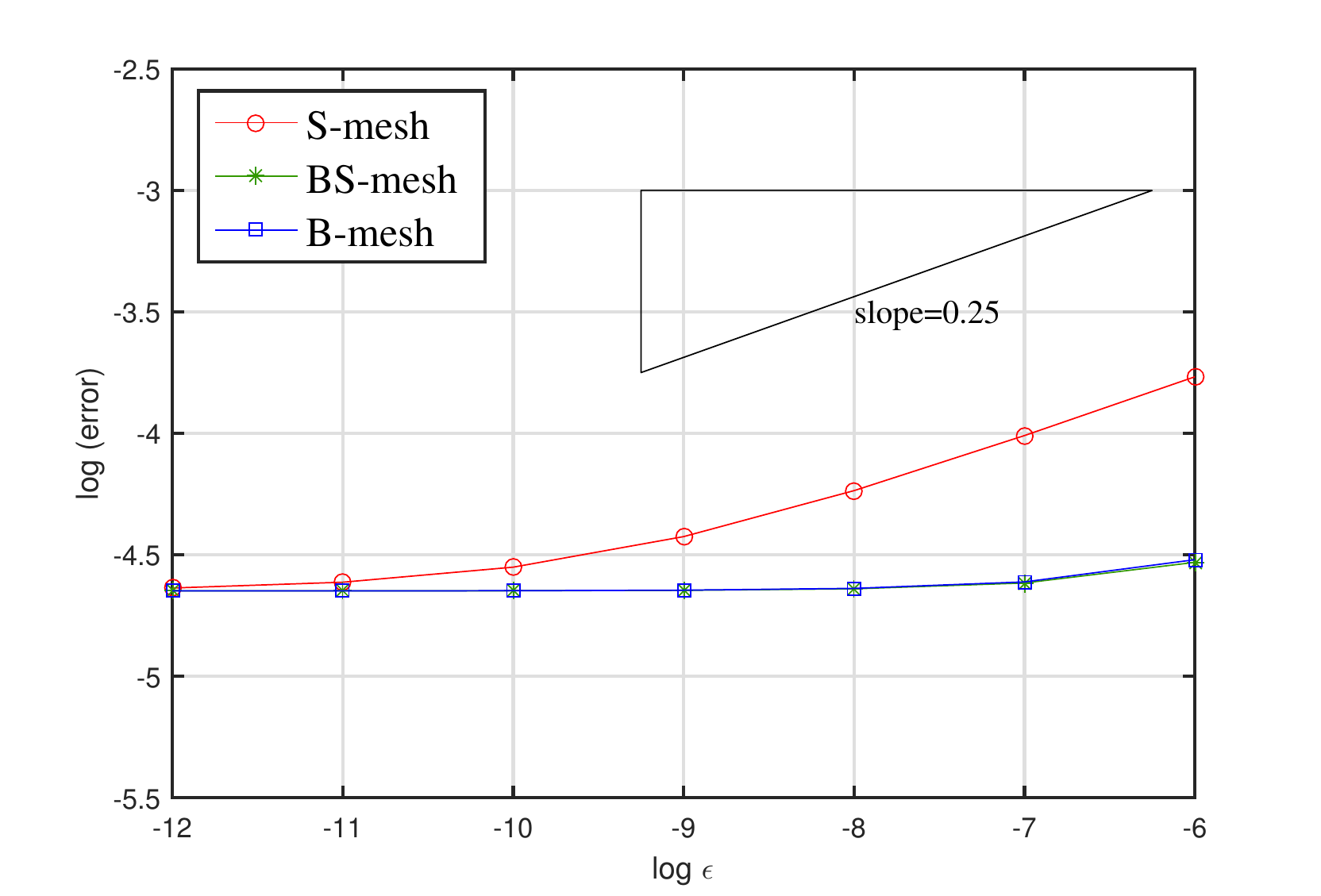}
	\caption{Energy norm error from $\sq$ in Example 1.}
	\label{fig:ep:err:linear}
\end{figure}

\begin{table}[ht]
	\normalsize
	\centering
	\caption{Balanced error and convergence rates for Example 2.}
	\label{table:stb:balanced:nonlinear}
	\begin{tabular}{cccccccc}
		\toprule
		\multirow{2}*{$k$}&\multirow{2}*{$N$} & \multicolumn{2}{c}{S-mesh} & \multicolumn{2}{c}{BS-mesh} & \multicolumn{2}{c}{B-mesh} \\
		\cmidrule(r){3-4} \cmidrule(r){5-6} \cmidrule(r){7-8}
		&& Balanced error & $r_s$ & Balanced error & $r_2$ & Balanced error & $r_2$
		\\
		\midrule
		0
		& 8 & 1.32e+00 & - & 1.30e+00 & - & 1.63e+00 & - \\
		& 16 & 1.09e+00 & 0.48 & 9.81e-01 & 0.41 & 1.10e+00 & 0.57 \\
		& 32 & 8.83e-01 & 0.45 & 7.07e-01 & 0.47 & 7.47e-01 & 0.56 \\
		& 64 & 6.99e-01 & 0.46 & 5.01e-01 & 0.50 & 5.15e-01 & 0.54 \\
		& 128 & 5.42e-01 & 0.47 & 3.54e-01 & 0.50 & 3.59e-01 & 0.52 \\
		& 256 & 4.13e-01 & 0.49 & 2.50e-01 & 0.50 & 2.52e-01 & 0.51 \\
		1
		& 8 & 5.07e-01 & - & 3.33e-01 & - & 5.41e-01 & - \\
		& 16 & 3.12e-01 & 1.20 & 1.37e-01 & 1.28 & 1.72e-01 & 1.65 \\
		& 32 & 1.68e-01 & 1.32 & 5.27e-02 & 1.38 & 5.88e-02 & 1.55 \\
		& 64 & 8.25e-02 & 1.40 & 1.96e-02 & 1.43 & 2.06e-02 & 1.51 \\
		& 128 & 3.79e-02 & 1.44 & 7.11e-03 & 1.46 & 7.29e-03 & 1.50 \\
		& 256 & 1.67e-02 & 1.47 & 2.55e-03 & 1.48 & 2.58e-03 & 1.50 \\
		2
		& 8 & 2.27e-01 & - & 1.01e-01 & - & 2.05e-01 & - \\
		& 16 & 1.05e-01 & 1.91 & 2.22e-02 & 2.19 & 3.18e-02 & 2.69 \\
		& 32 & 3.80e-02 & 2.16 & 4.37e-03 & 2.34 & 5.22e-03 & 2.61 \\
		& 64 & 1.16e-02 & 2.32 & 8.19e-04 & 2.42 & 8.93e-04 & 2.55 \\
		& 128 & 3.15e-03 & 2.42 & 1.50e-04 & 2.45 & 1.56e-04 & 2.52 \\
		& 256 & 7.95e-04 & 2.46 & 2.69e-05 & 2.47 & 2.74e-05 & 2.51 \\
		3
		& 8 & 1.01e-01 & - & 3.06e-02 & - & 8.30e-02 & - \\
		& 16 & 3.57e-02 & 2.57 & 3.56e-03 & 3.10 & 5.96e-03 & 3.80 \\
		& 32 & 8.90e-03 & 2.96 & 3.61e-04 & 3.30 & 4.66e-04 & 3.68 \\
		& 64 & 1.71e-03 & 3.23 & 3.43e-05 & 3.40 & 3.87e-05 & 3.59 \\
		& 128 & 2.78e-04 & 3.38 & 3.15e-06 & 3.45 & 3.33e-06 & 3.54 \\
		& 256 & 4.03e-05 & 3.45 & 2.84e-07 & 3.47 & 2.91e-07 & 3.51 \\
		\bottomrule
	\end{tabular}
\end{table}

\begin{table}[ht]
	\normalsize
	\centering
	\caption{Energy error and convergence rates for Example 2.}
	\label{table:stb:energy:nonlinear}
	\begin{tabular}{cccccccc}
		\toprule
		\multirow{2}*{$k$}&\multirow{2}*{$N$} & \multicolumn{2}{c}{S-mesh} & \multicolumn{2}{c}{BS-mesh} & \multicolumn{2}{c}{B-mesh} \\
		\cmidrule(r){3-4} \cmidrule(r){5-6} \cmidrule(r){7-8}
		&& Energy error & $r_S$ & Energy error & $r_2$ & Energy error & $r_2$
		\\
		\midrule
		0
		& 8 & 1.06e-02 & - & 9.30e-03 & - & 1.51e-02 & - \\
		& 16 & 8.15e-03 & 0.65 & 5.57e-03 & 0.74 & 7.65e-03 & 0.98 \\
		& 32 & 5.79e-03 & 0.73 & 3.12e-03 & 0.84 & 3.96e-03 & 0.95 \\
		& 64 & 3.85e-03 & 0.80 & 1.69e-03 & 0.89 & 2.05e-03 & 0.95 \\
		& 128 & 2.41e-03 & 0.86 & 8.95e-04 & 0.92 & 1.05e-03 & 0.96 \\
		& 256 & 1.45e-03 & 0.91 & 4.69e-04 & 0.95 & 5.36e-04 & 0.97 \\
		1
		& 8 & 5.07e-03 & - & 3.29e-03 & - & 6.42e-03 & - \\
		& 16 & 2.86e-03 & 1.41 & 1.12e-03 & 1.56 & 1.76e-03 & 1.86 \\
		& 32 & 1.37e-03 & 1.57 & 3.35e-04 & 1.74 & 4.68e-04 & 1.91 \\
		& 64 & 5.73e-04 & 1.70 & 9.48e-05 & 1.82 & 1.23e-04 & 1.93 \\
		& 128 & 2.17e-04 & 1.80 & 2.60e-05 & 1.87 & 3.18e-05 & 1.95 \\
		& 256 & 7.59e-05 & 1.88 & 6.97e-06 & 1.91 & 8.19e-06 & 1.96 \\
		2
		& 8 & 2.27e-03 & - & 9.73e-04 & - & 2.34e-03 & - \\
		& 16 & 9.63e-04 & 2.11 & 1.75e-04 & 2.47 & 3.11e-04 & 2.91 \\
		& 32 & 3.15e-04 & 2.37 & 2.71e-05 & 2.69 & 4.08e-05 & 2.93 \\
		& 64 & 8.49e-05 & 2.57 & 3.89e-06 & 2.80 & 5.31e-06 & 2.94 \\
		& 128 & 1.96e-05 & 2.72 & 5.36e-07 & 2.86 & 6.86e-07 & 2.95 \\
		& 256 & 4.04e-06 & 2.82 & 7.23e-08 & 2.89 & 8.80e-08 & 2.96 \\
		3
		& 8 & 1.00e-03 & - & 2.91e-04 & - & 9.59e-04 & - \\
		& 16 & 3.27e-04 & 2.76 & 2.80e-05 & 3.38 & 6.15e-05 & 3.98 \\
		& 32 & 7.46e-05 & 3.14 & 2.23e-06 & 3.65 & 3.87e-06 & 3.99 \\
		& 64 & 1.29e-05 & 3.43 & 1.62e-07 & 3.78 & 2.43e-07 & 3.99 \\
		& 128 & 1.83e-06 & 3.63 & 1.13e-08 & 3.85 & 1.53e-08 & 3.99 \\
		& 256 & 2.21e-07 & 3.77 & 7.62e-10 & 3.89 & 9.56e-10 & 4.00 \\
		\bottomrule
	\end{tabular}
\end{table}

\begin{table}[htp]
	\normalsize
	\centering
	\caption{ Energy norm and balanced-norm errors
		for different $\varepsilon$.}
	\label{table:eps:nonlinear}
	\begin{tabular}{ccccccc}
		\toprule
		& \multicolumn{3}{c}{Energy error} &\multicolumn{3}{c}{Balanced error}\\
		\cmidrule(r){2-4} \cmidrule(r){5-7}
		$\varepsilon$ & S-mesh & BS-mesh & B-mesh & S-mesh & BS-mesh & B-mesh\\
		\midrule
		$10^{-6}$  & 2.40e-04 & 2.20e-05 & 2.37e-05
		& 1.67e-02 & 2.55e-03 & 2.57e-03\\
		$10^{-7}$  & 1.35e-04 & 1.24e-05 & 1.40e-05
		& 1.67e-02 & 2.55e-03 & 2.58e-03\\
		$10^{-8}$  & 7.59e-05 & 6.97e-04 & 8.19e-06
		& 1.67e-02 & 2.55e-03 & 2.58e-03\\
		$10^{-9}$  & 4.27e-05 & 3.92e-06 & 4.71e-06
		& 1.67e-02 & 2.55e-03 & 2.59e-03\\
		$10^{-10}$  & 2.40e-05 & 2.21e-06 & 2.69e-06
		& 1.67e-02 & 2.55e-03 & 2.59e-03\\
		$10^{-11}$  & 1.35e-05 & 1.24e-06 & 1.52e-06
		& 1.67e-02 & 2.55e-03 & 2.59e-03\\
		$10^{-12}$  & 7.58e-06 & 6.98e-07 & 8.55e-07
		& 1.67e-02 & 2.55e-03 & 2.59e-03\\
		$10^{-13}$  & 4.27e-06 & 3.92e-07 & 4.79e-07
		& 1.67e-02 & 2.55e-03 & 2.59e-03\\
		$10^{-14}$  & 2.40e-06 & 2.21e-07 & 2.69e-07
		& 1.67e-02 & 2.55e-03 & 2.59e-03\\
		$10^{-15}$  & 1.35e-06 & 1.24e-07 & 1.50e-07
		& 1.67e-02 & 2.55e-03 & 2.59e-03\\
		$10^{-16}$  & 7.58e-07 & 6.98e-08 & 8.40e-08
		& 1.67e-02 & 2.55e-03 & 2.59e-03\\
		\bottomrule
	\end{tabular}
\end{table}

\begin{figure}
	\centering
	\includegraphics[width=0.7\linewidth]{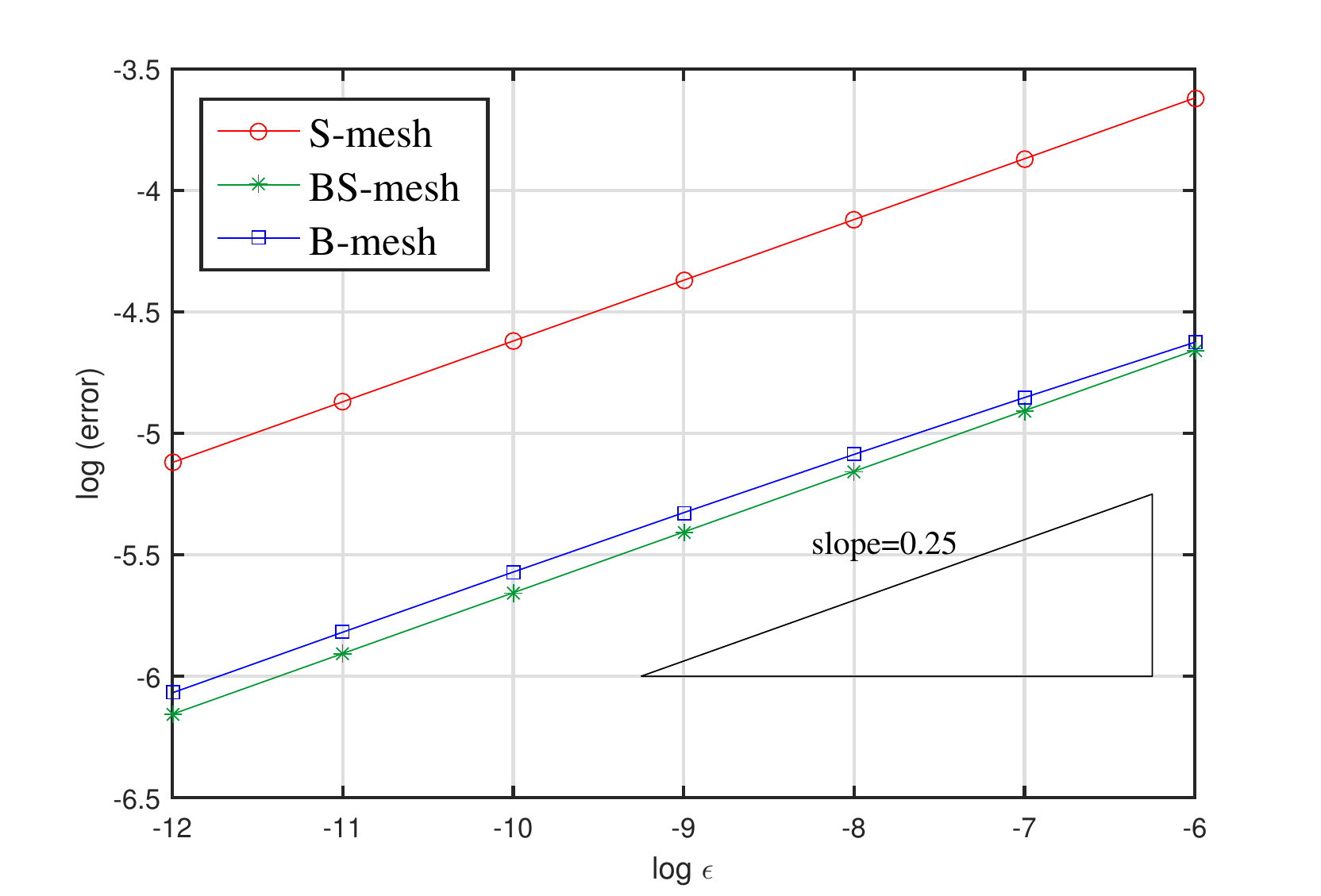}
	\caption{Energy norm error from $\sq$ in Example 2.}
	\label{fig:ep:err:nonlinear}
\end{figure}

\section*{Acknowledgements}

This study was supported by the National Natural Science Foundation of China 
(No. 11801396), and the Natural Science Foundation of Jiangsu Province
(No. BK20170374).

\section*{Appendix}
\label{sec:LDGCN}
\setcounter{equation}{0}
\setcounter{subsection}{0}
\renewcommand{\theequation}{A.\arabic{equation}}
\renewcommand{\thetheorem}{A.\arabic{theorem}}
\renewcommand{\thesubsection}{A.\arabic{subsection}}
\renewcommand{\theproposition}{A.\arabic{proposition}}

In this appendix, we mention several results for parabolic singularly
perturbed reaction--diffusion problems:
\begin{subequations}\label{spp:dvp:R-D}
	\begin{align}
	u_t -\varepsilon \Delta u +  b(x,y) u
	&= f(x,y,t)   &\textrm{in}&\;  \Omega \times (0,T],
	\\
	u|_{t=0} &= u_0(x,y),   &\textrm{in}&\; \overline{\Omega}
	\\
	u|_{\partial\Omega} &= 0,   &\textrm{for}&\; t \in (0,T).
	\end{align}
\end{subequations}

Let $M$ be a positive integer and $0=t^0<t^1<\dots<t^M=T$ 
be an equidistant partition of $[0,T]$.
We define the time interval $K^m=(t^{m-1},t^m]$, $m=1,2,\dots,M$,
with a mesh width $\Delta t=t^m-t^{m-1}$, which satisfies $M\Delta t =T$.
We write $v^{m}=v(t^m)$ and $v^{m,\theta}=\theta v^{m}+(1-\theta)v^{m-1}$.

Assume $1/2\leq \theta \leq 1$.
The fully discrete scheme for \eqref{spp:dvp:R-D}
is constructed by the LDG in space and 
by the implicit $\theta$-scheme in time.
It reads as follows:
Let $\uph^0 = \Pi u_0$ be the local $L^2$ projection of $u_0$.
For any $m=1,2,\dots,M$, find the numerical solution
$\wph^{m}=(\uph^m,\pph^m,\qph^m)\in \mathcal{V}_N^3$ such that
\begin{align}\label{fully:discrete}
\dual{\frac{\uph^{m}-\uph^{m-1}}{\Delta t}}{\vphu}
+\Bln{\wph^{m,\theta}}{\vph}=0
\end{align}
holds for any $\vph=(\vphu,\vphp,\vphq)\in \spc^3$,
where $\Bln{\wph^{m,\theta}}{\vph}$ is defined in \eqref{B:def:2d}.
Here, we use the abbreviations
$t^{m,\theta}=\theta t^m+(1-\theta)t^{m-1}$ and 
$g^{m,\theta}=\theta g^m+(1-\theta)g^{m-1}$
for functions $g=g(t)$ and $g^m=g(m\Delta t)$.
Following the traditional energy analysis and concept of \cite{Cheng2020},
we have

\begin{theorem}
	Let $\bm w=(u,p,q)$ be the solution to problem \eqref{spp:dvp:R-D}
	and satisfy	an analogous decomposition as the stationary case.
	Let $\wph^{m}=(\uph^m,\pph^m,\qph^m)\in \spc^3$,
	$m=1,2,\dots,M$ be the numerical solution
	of the fully discrete scheme \eqref{fully:discrete},
	where $V_N$ is composed of piecewise polynomials of degree $k\geq 0$
	on layer-adapted meshes \eqref{layer-adapted}
	with $\sigma\geq k+2$.
	Then, there exists a constant $C>0$ independent of $\varepsilon,N$ and $M$ such that
	\begin{align}\label{err:estimate:LDGCN}
	&\norm{u^M-\uph^M}^2
	+\Delta t\sum_{m=1}^M \enorm{\bm{w}^{m,\theta}-\wph^{m,\theta}}^2 
	\nonumber\\
	&\leq C(1+T)\Big[
	(\ssq+\parep)(N^{-1}\max|\ppri|)^{2(k+1)}
	+N^{-2(k+1)}
	+ (\Delta t)^{2r}\Big],
	\end{align}
	where $r=1$ if $1/2<\theta\leq 1$ and $r=2$ if $\theta=1/2$.
\end{theorem}

%

\end{document}